%% file: CombLensSpaceKnot30.tex
\documentclass[10pt]{amsart}
\theoremstyle{plain}
\usepackage{amsmath, amsfonts, epsfig}
\usepackage{graphics}
\usepackage{color}
\usepackage[all]{xy}
\DeclareMathOperator{\Int}{Int}
\DeclareMathOperator{\Sym}{Sym}
\DeclareMathOperator{\PD}{PD}

\graphicspath{ {Figures/} }

\newtheorem{theorem}{Theorem}[section]
\newtheorem{lemma}[theorem]{Lemma}
\newtheorem{proposition}[theorem]{Proposition}
\newtheorem{corollary}[theorem]{Corollary}
\newtheorem{conjecture}[theorem]{Conjecture}
\newtheorem{remark}[theorem]{Remark}
\theoremstyle{definition}
\newtheorem{definition}[theorem]{Definition}

\newcommand{\ons}{{Ozsv{\'a}th} and {Szab{\'o}} }

\newcommand{\A}{ {\bf A}}
\newcommand{\Af}{ {\bf A}^{\mathfrak{f}}}
\newcommand{\M}{ {\bf M}}
\newcommand{\Mf}{ {\bf M}^{\mathfrak{f}}}
\newcommand{\gr}{{\Mf}} 
\newcommand{\wtgr}{{\Mf}} 
\newcommand{\Sp}{ {\bf S}}
\newcommand{\Sf}{ {\bf S}^{\mathfrak{f}}}

\newcommand{\Z}{\mathbb{Z}}
\newcommand{\R}{\mathbb{R}}
\newcommand{\Q}{\mathbb{Q}}
\newcommand{\XX}{\mathbb{X}}
\newcommand{\OO}{\mathbb{O}}

\newcommand{\bdry}{\ensuremath{\partial}}
\newcommand{\T}{\mathbb{T}}
\renewcommand{\t}{\ensuremath{\mathfrak{t}}}

\newcommand{\s}{\ensuremath{\mathfrak{s}}}
\newcommand{\us}{\ensuremath{{\underline{\mathfrak{s}}}}}

\newcommand{\Spin}{\ensuremath{{\mbox{Spin}}}}
\newcommand{\uSpin}{\underline{\ensuremath{{\mbox{Spin}}}}}
\newcommand{\x}{\ensuremath{{\bf x}}}
\newcommand{\y}{\ensuremath{{\bf y}}}
\newcommand{\Par}{P}
\newcommand{\PG}{PG}

\title[]{Grid Diagrams for lens spaces and combinatorial knot Floer homology}
\author[Kenneth L.\ Baker]{Kenneth L.\ Baker}
\address{Kenneth L.\ Baker \newline\indent School of Mathematics \newline\indent Georgia Institute of Technology \newline\indent Atlanta, Georgia 30332}
\email{\rm{kb@math.gatech.edu}}
\author[J. Elisenda Grigsby]{J. Elisenda Grigsby}
\address{J. Elisenda Grigsby\newline\indent Department of Mathematics\newline\indent Columbia University\newline\indent 2990 Broadway MC4406\newline\indent NY, NY 10027}
\email{\rm{egrigsby@math.columbia.edu}}
\author[Matthew Hedden]{Matthew Hedden}
\address{Matthew Hedden\newline\indent Department of Mathematics\newline\indent Massachusetts Institute of Technology\newline\indent Building 2, Room 236\newline\indent 77 Massachusetts Avenue\newline\indent Cambridge, MA 02139-4307}
\email{\rm{mhedden@math.mit.edu}}

\begin{document}

\begin{abstract} Similar to knots in $S^3$, any knot in a lens space has a grid diagram from which one can combinatorially compute all of its knot Floer homology invariants.  We give an explicit description of the generators, differentials, and rational Maslov and Alexander gradings in terms of combinatorial data on the grid diagram.  Motivated by existing results for the Floer homology of knots in $S^3$ and the similarity of the resulting combinatorics presented here, we conjecture that a certain family of knots is characterized by their Floer homology.  Coupled with work of the third author, an affirmative answer to this would prove the Berge conjecture, which catalogs the knots in $S^3$ admitting lens space surgeries. 
\end{abstract}

\maketitle

\section{Introduction}\label{section:Introduction}
The Heegaard Floer homology package is a powerful collection of invariants of knots, links, and $3$-- and $4$--manifolds.  Although the generators of the chain complexes  used to define these invariants are combinatorial, the differential involves a count of J-holomorphic curves in a symplectic manifold.  Until recently, there was no combinatorial method to perform this count.

In 2006, Sarkar made the revolutionary observation that if a particular type of Heegaard diagram could be found, a general enumeration of the holomorphic curves counted by the resulting chain complex or chain map was possible.  To describe this method, recall that a pointed Heegaard diagram is a collection of data:
$$(\Sigma,\vec{\alpha}=\{\alpha_1,\ldots,\alpha_k\},\vec{\beta}=\{\beta_1,\ldots,\beta_k\}, \vec{z}=\{z_1,\ldots,z_l\}),$$
where $\Sigma$ is a closed Heegaard surface for a three-manifold $Y$, $\vec{\alpha},\vec{ \beta}\subset \Sigma$ are collections of simple closed curves bounding disks in the Heegaard splitting, and $\vec{ z}\subset \Sigma$ is a collection of points which can be used to specify knots and links in $Y$.
Sarkar showed that if all the connected components (regions) of $\Sigma -\vec{\alpha} -\vec{\beta}$ are $2$-- or $4$--sided polygons then any holomorphic curve relevant to the differential could be identified combinatorially,
via a formula of Lipshitz \cite{Lipshitz}.

In this level of generality it is clear, however, that the idea cannot work for every example.  This is because the sphere and torus are the only surfaces which can be decomposed as cell complexes consisting solely of polygons with $2$ or $4$ sides, while the only $3$--manifolds with Heegaard genus $0$ or $1$ are the lens spaces.  In \cite{GT0607777}, Sarkar and Wang exhibit an algorithm to compute a restricted version (the so-called ``hat'' theory) of the knot Floer homology for an arbitrary $3$--manifold by finding a Heegaard diagram with a single ``bad'' region (many-sided polygon) which is not counted in the differential.

On the other hand, restricting to knots in lens spaces allows for a computation of the full range of invariants.

 In the case of $S^3$, this approach was taken in a paper of Manolescu, Ozsv{\'a}th, and Sarkar \cite{GT0607691}.  In addition to allowing for the computation of the filtered chain homotopy type of $CF^{-}(S^3,K)$ - the most robust of the knot and link invariants - their approach was ground-breaking in its combinatorial simplicity.  Indeed, generators are identified with elements of $S_n$, the symmetric group on $n$ letters, and differentials count domains associated to elements differing by simple transpositions.  Possible connections to representation theory are tantalizing.

The purpose of this paper is to extend the combinatorial algorithm for computing the Floer homology of knots in $S^3$ in terms of grid diagrams to the case of knots in an arbitrary lens space (we do not treat the case of $S^1\times S^2$ here, as the setup is quite different).  In fact, the algorithm of \cite{GT0607691} carries over in a somewhat straightforward manner.  In Section~\ref{section:GridDiagrams} we describe grid position for knots in lens spaces and show how to pass from these grid positions to the relevant Heegaard diagrams, which we call {\em twisted toroidal grid diagrams} or simply {\em grid diagrams}.  These are the analogues of the toroidal grid diagrams used by \cite{GT0607691} to compute the Floer homology of knots in $S^3$.   The main theorem of this paper is the following:

\begin{theorem} \label{thm:main}
Let $K \subset L(p,q)$ be an arbitrary knot in a lens space.   Then $K$ may be put into grid position and admits an associated twisted toroidal grid diagram.  Moreover the filtered chain homotopy type of $CF^-(L(p,q),K)$ as a $\mathbb{Z}_2[U]$-module can be computed in terms of the combinatorics of the grid diagram.
\end{theorem}

We postpone a precise description of the resulting chain complex until the next section, but content ourselves here to say that it closely resembles the chain complex for knots in $S^3$ given in \cite{GT0607691}.  Instead, we briefly discuss our motivation for the present generalization.  

The most obvious motivation for pursuing a combinatorial formula for the knot Floer homology invariants of lens space knots is the strength of these invariants. Link Floer homology in manifolds other than $S^3$ has been shown to detect the Thurston norm of a link complement \cite{GT0601618},\cite{GT0604360}, whether a knot is fibered \cite{GT0607156} and has had applications to questions related to the concordance classes of knots in $S^3$ through the double-branched covering operation \cite{GT0508065, GT0611023, GT0701460, GT07063398}.

Our main motivation, however, is in providing a foundation for a combinatorial approach to the Berge conjecture.   In \cite{Berge}, Berge describes a family of knots, which he calls {\em double primitive knots}, on which one can perform Dehn surgery and obtain a lens space.  These knots are characterized by the property that they can be embedded in the Heegaard surface of a genus two Heegaard splitting of $S^3$ in such a way that they represent a generator of the fundamental group of each handlebody of the splitting.  The Berge Conjecture is that any knot in $S^3$ which admits a lens space surgery is double primitive, \cite{Berge}. 

By shifting our perspective to the lens spaces, one can transform this conjecture into an infinite number of simpler conjectures i.e.\ we can try to prove that any knot in $S^3$ on which surgery can produce a particular lens space, say $L(43,5)$, is double primitive.  To make this more precise, note that upon performing the surgery on a double primitive knot $K\subset S^3$ which yields a lens space, there is a naturally induced knot $K'\subset L(p,q)$ in the lens space.  This knot is the core of the solid torus glued to $S^3-K$ during the surgery.  Berge showed that if $K\subset S^3$ is double primitive then the induced knot $K'\subset L(p,q)$ is a member of a particularly simple finite family.  In fact, $K'$ must be one of the $p$ knots which can be realized by a grid diagram of grid number one.  Furthermore, if surgery on a grid number one knot $K\subset L(p,q)$ yields the three-sphere, then the knot $K' \subset S^3$ induced by the surgery is double primitive.   (Let us say a knot {\em has grid number one} if it may be represented by a grid number one grid diagram.)   Thus, we have the following equivalent form of the Berge conjecture, which was originally stated as a question \cite{Berge}.

\begin{conjecture}\cite{Berge} Suppose that surgery on $K \subset L(p,q)$ yields the three-sphere.  Then $K$ has grid number one.
\end{conjecture}

One immediately observes that grid number one knots have ``simple knot Floer homology'': $$rk(\widehat{HFK}(L(p,q),K)) = rk(\widehat{HF}(L(p,q)).$$
Since there is a spectral sequence starting with knot Floer homology and converging to the Floer homology of the ambient three-manifold, the above equality can be informally described as ``grid number one knots have the smallest rank knot Floer homology possible''.

We are then led to consider the following strategy for proving the Berge Conjecture:
\begin{enumerate}
	\item Show that if surgery on $K\subset L(p,q)$ yields $S^3$, then $K$ has simple Floer homology in the above sense.
	\item Show that if $K\subset L(p,q)$ has simple Floer homology, then $K$ has grid number one.  
\end{enumerate}
Thus, the strategy can be described succinctly as 
$$ K\text{ has simple surgery} \implies K \text{ has simple knot Floer homology} \implies K\text{ is a simple knot}.$$

While at first sight this strategy may appear overly optimistic, we note that the first step has been carried out by the third author \cite{Hedden}, and independently by Rasmussen \cite{Rasmussen}: 
\begin{theorem}\label{thm:simpleHFK}\cite{Hedden,Rasmussen} Suppose that surgery on $K\subset L(p,q)$ yields $S^3$ and let $g(K)$ denote the genus of $K$. Then $p\ge 2g(K)-1$. Furthermore, \begin{itemize}
		\item If $p> 2g(K)-1$ then $rk(\widehat{HFK}(L(p,q),K)) = rk(\widehat{HF}(L(p,q)).$
		\item If $p= 2g(K)-1$ then $rk(\widehat{HFK}(L(p,q),K)) = rk(\widehat{HF}(L(p,q))+2.$
	\end{itemize}
		\end{theorem}
\begin{remark} In order for surgery on $K$ to produce $S^3$, $K$ must generate $H_1(L(p,q);\Z)\cong\Z/p\Z$.  As $K$ is not null-homologous, we should be careful to say what we mean by the genus.  Since surgery on $K$ produces $S^3$, the complement $L(p,q)-K$ is homeomorphic to the complement $S^3-K'$, for a knot $K'$ in $S^3$.  We define $g(K)$ to be the Seifert genus of $K'$.
\end{remark}
	
The Berge conjecture would then follow from:

\begin{conjecture}\label{conj:gn1}
Suppose that $K\subset L(p,q)$ satisfies $$rk(\widehat{HFK}(L(p,q),K)) = p.$$ Then $K$ has grid number $1$.
\end{conjecture}
\begin{conjecture}\label{conj:T} There are exactly two knots in $L(p,q)$ which satisfy  $$rk(\widehat{HFK}(L(p,q),K)) = p+2.$$
\end{conjecture}

In Section $4$ of \cite{Hedden}, two knots $T_1,T_2\subset L(p,q)$ satisfying $rk(\widehat{HFK}(L(p,q),T_i)) = p+2$ are specified for each lens space.  There, it is shown that surgery on $T_i$ cannot produce $S^3$. Thus a proof of the above conjectures, together with Theorem~\ref{thm:simpleHFK}, would indeed prove the Berge conjecture.

Though our Conjectures are quite strong, we note that in the case $L(p,q)=S^3$, we have affirmative answers to the first and a specialization of the second. 

\begin{theorem}\label{thm:S3} \cite{MR2023281}
Suppose $K \subset S^3$ satisfies $rk(\widehat{HFK}(S^3,K)) = 1$. Then $K$ is the unknot (the only grid number one knot in $S^3$). 
\end{theorem}
\begin{theorem}\label{thm:Ghiggini}\cite{Ghiggini}
Suppose $K \subset S^3$ satisfies  $rk(\widehat{HFK}(S^3,K)) = 3$ and $g(K) = 1$.  Then $K$ is the right- or left-handed trefoil.
\end{theorem}

The proofs of the above theorems rely on connections between Heegaard Floer homology and symplectic geometry and have  yet to be understood combinatorially.
Such an understanding of Theorems~\ref{thm:S3} and~\ref{thm:Ghiggini} would likely lead to a proof of our conjectures and, hence, of the Berge conjecture.

\begin{remark}  We find it worthwhile to remark that while \cite{GT0607777} provides an algorithm for  computing  $\widehat{CF}(L(p,q),K)$,  implementation varies on a knot-by-knot basis and is time-consuming in practice.  The present work has the advantage that the chain complexes  are explicit, combinatorial in description, and uniformly implemented.  Another key feature  is that  a combinatorial description of the filtered chain homotopy type of $CF^-(L(p,q),K)$ is provided.  This invariant contains strictly more information than the filtered chain homotopy type of $\widehat{CF}(L(p,q),K)$ and, in particular,  is required for formulas which compute the Floer homology of closed three-manifolds obtained by Dehn surgery along $K\subset L(p,q)$ \cite{GT0504404}.
\end{remark}

\subsection{Outline}

In the next section we define grid diagrams for knots and links in lens spaces.  To a grid diagram for a knot $K$, $G_K$, we associate a module $C^-(G_K)$, equipped with an endomorphism $\partial^-$, both of which are defined in terms of the combinatorics of these diagrams.   We also associate to elements in $C^-(G_K)$ three combinatorial quantities $(\Sp,\M,\A)$.  Theorem~\ref{thm:main} states that the object we define is a chain complex for the knot Floer homology, and that $(\Sp,\M,\A)$ coincide with the $\Spin^c$, Maslov (homological), and Alexander gradings on knot Floer homology, respectively.  The proof will be divided into a series of propositions:  
	\begin{itemize}
		\item Proposition~\ref{prop:Boundary} identifies the ungraded combinatorial object $(C^-(G_K),\partial^-)$ with a chain complex for the  knot Floer homology.
		\item	Proposition~\ref{prop:SpincGrad} equates the combinatorial quantity $\Sp\in \Z_p$ with the $\Spin^c$ grading $\Sf$ on Floer homology.
		\item  Proposition~\ref{prop:MaslovGrad} equates the combinatorial quantity $\M\in \Q$ with the Maslov grading $\Mf$ on Floer homology.
		\item Proposition~\ref{prop:AlexGrad} equates the combinatorial quantity $\A\in \Q$ with the Alexander grading $\Af$ on knot Floer homology.	
	\end{itemize}
 
That the knot Floer homology can be computed from grid diagrams will be more or less straightforward, and Proposition~\ref{prop:Boundary} will follow in the spirit of the analogous theorem for knots in $S^3$ \cite{GT0607691}.  The bulk of the work will be in showing that the three combinatorial gradings agree with the $\Spin^c$,  Alexander, and Maslov gradings, originally defined using vector fields and index theory.

	In Section $3$, we recall necessary background on Heegaard Floer theory, paving the way for a proof of Theorem~\ref{thm:main}  in Section $4$. In particular, Section $4$ contains proofs of the above propositions.  Additionally, we prove there (see Proposition~\ref{prop:existence} and its corollary) that every knot in a lens space possesses a grid diagram.   An important technical tool in the proofs will be a correspondence between grid diagrams for knots in lens spaces and grid diagrams for certain links in $S^3$ - the universal cover of $L(p,q)$ (this correspondence was developed by the second author and her collaborators in \cite{GT0701460}).  Indeed, under the covering projection $\pi:S^3\rightarrow L(p,q)$, a grid diagram for $(L(p,q),K)$ lifts to a grid diagram for a knot $\tilde{K}\subset S^3$, and this lifted diagram can be used together with results of \cite{GT0610559} and \cite{GT0608001} to help establish our grading formulas.   

\subsection{Acknowledgments} The first author was partially supported by NSF Grant DMS--0239600.  The second author was partially supported by an NSF postdoctoral fellowship. The third author was partially supported by NSF Grant DMS-0706979.
   
\section{Combinatorial Description of $(CF^-(L(p,q),K), \bdry^-)$}\label{section:CombSummary}
In this section we provide a purely combinatorial description of the Heegaard Floer invariants of knots in lens spaces, making no mention of $J$--holomorphic curves.   We postpone the proof that our chain complex is isomorphic to the one defined by \ons in \cite{GT0504404} (see also \cite{MR2065507} and \cite{GT0306378})  until Section~\ref{section:GridDiagrams}, after having reviewed the relevant aspects  of Heegaard Floer theory in Section~\ref{section:HFBackground}.

Throughout, we assume $p,q \in \mathbb{Z}$ are relatively prime, with $p \in \mathbb{Z}_+, q \neq 0$, and $-p < q < p$.  $L(p,q)$ denotes $-\frac{p}{q}$ surgery on the unknot in $S^3$.  The notation ``$a \mod n$'', for $n \in \Z_+$ and $a \in \R$, refers to the unique element of the set $\{a + kn| k \in \Z\}$ in the range $[0,n)$.

Isotopy classes of knots and links in $L(p,q)$ are encoded in the combinatorics of {(twisted toroidal) grid diagrams}, which we now define: 

\begin{definition} \label{def:TTGridDiag}  A {\em (twisted toroidal) grid diagram} $G_K$ with grid number $n$ for $L(p,q)$ consists of  a five-tuple $(T^2,\vec{\alpha},\vec{\beta},\vec{\mathbb{O}},\vec{\mathbb{X}})$ (illustrated in Figure~\ref{fig:twistedgrid}), where: 
  \begin{itemize}
	  \item $T^2$ is the standard oriented torus $\R^2 / \Z^2$, identified with the quotient of $\mathbb{R}^2$ (with its standard orientation) by the $\mathbb{Z}^2$ lattice generated by the vectors $(1,0)$ and $(0,1)$. \vspace{2mm}
    \item $\vec{\alpha} = \{\alpha_0, \ldots, \alpha_{n-1}\}$ are the $n$ images $\alpha_i$ in $T^2 = \mathbb{R}^2/\mathbb{Z}^2$ of the lines $y = \frac{i}{n}$ for $i \in \{0, \ldots n-1\}$.  Their complement $T^2-\alpha_0 - \ldots - \alpha_{n-1}$ has $n$ connected annular components, which we call the {\em rows} of the grid diagram.  \vspace{2mm}
    \item $\vec{\beta} = \{\beta_0, \ldots, \beta_{n-1}\}$ are the $n$ images $\beta_i$ in $T^2 = \mathbb{R}^2/\mathbb{Z}^2$ of the lines $y = -\frac{p}{q}(x-\frac{i}{pn})$ for $i \in \{0, \ldots n-1\}$.  Their complement $T^2 - \beta_0 - \ldots - \beta_{n-1}$ has $n$ connected annular components, which we call the {\em columns} of the grid diagram.  \vspace{2mm} 
    \item $\vec{\mathbb{O}} = \{O_0, \ldots, O_{n-1}\}$ are $n$ points in $T^2 - \vec{\alpha} - \vec{\beta}$ with the property that no two $O$'s lie in the same row or column.  \vspace{2mm}
    \item $\vec{\mathbb{X}} = \{X_0, \ldots, X_{n-1}\}$ are $n$ points in $T^2 - \vec{\alpha} - \vec{\beta}$ with the property that no two $X$'s lie in the same row or column. \vspace{2mm}
  \end{itemize}
\end{definition}

\begin{figure}
\begin{center}
\input{Figures/TwistedGrid-new2.pstex_t}
\end{center}
\caption{The {\it preferred fundamental domain} on $\mathbb{R}^2$ describing a twisted toroidal grid diagram $G_K$ with grid number $n=4$ for a link $L$ in $L(5,2)$.  Here, $C_3$ is one of the four columns, while $R_i$ are the rows. Throughout the paper, we will use this fundamental domain for $G_K$, which consists of $(x,y) \in \R^2$ satisfying $0 \leq y < 1$, $-\frac{q}{p}y \leq x < -\frac{q}{p}y + 1$.}
\label{fig:twistedgrid}
\end{figure}
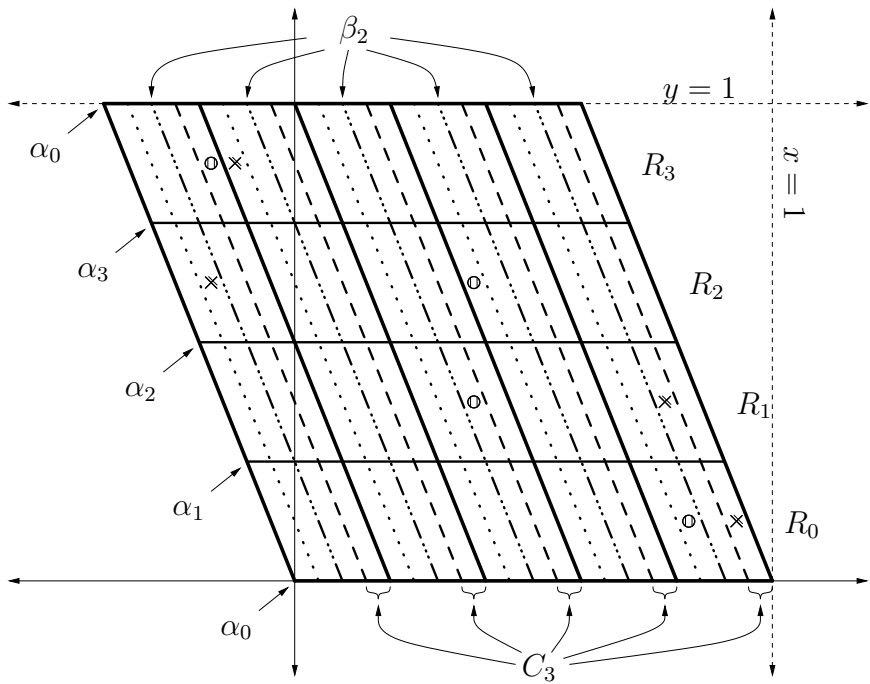

The lens space $L(p,q)$ is comprised of two solid tori, $V_\alpha$ and $V_\beta$, with common boundary $T^2$.  We view $V_\alpha$ as below $T^2$ and $V_\beta$ as above $T^2$.  The curves in $\vec{\alpha}$ are meridians of $V_\alpha$ and the curves in $\vec{\beta}$ are meridians of $V_\beta$.  
A grid diagram $G_K$ uniquely specifies an oriented knot or link $K$ in $L(p,q)$ as follows:

\begin{enumerate}
\item First connect each $X_i$ to the unique $O_j$ lying in the same row as $X_i$ by an oriented ``horizontal'' arc embedded in that row of $T^2$, disjoint from the $\vec{\alpha}$ curves.
\item Next connect each $O_{j}$ to the unique $X_m$ lying in the same column as $O_j$ by an oriented ``slanted'' arc embedded in that column of $T^2$, disjoint from the $\vec{\beta}$ curves.\footnote{If an $O_i$ and an $X_j$ coincide, then we offset one basepoint slightly from the other and join them by two small arcs to form a trivial unknotted component.} 
\item The union of these two collections of $n$ arcs forms an immersed (multi)curve $\gamma$ in $T^2$.  Remove all self-intersections of $\gamma$ by pushing the interiors of the horizontal arcs slightly down into $V_\alpha$ and the interiors of the slanted arcs slightly up into $V_\beta$.
\end{enumerate}
		The above construction associates a unique isotopy class of oriented knot or link to a grid diagram.  Different choices of horizontal arcs in a row are equated by an isotopy rel--boundary within $V_\alpha$ that is disjoint from a set meridional disks bounded by the $\vec{\alpha}$ curves; similarly for the slanted arcs.  Furthermore, it is straightforward to see that any isotopy class of oriented links in $L(p,q)$ can be realized by a grid diagram (see Proposition~\ref{prop:existence} and its corollary below).

Pick, then, a grid diagram $G_K =(T^2, \vec{\alpha}, \vec{\beta}, {\OO}, {\XX})$ for an oriented knot $K\subset L(p,q)$.   We construct a  filtered, graded chain complex $(C^-(G_K), \bdry^-)$ associated to the grid diagram $G_K$ for this knot.\footnote{ For the ease of exposition, we restrict ourselves to the situation where $K$ is a knot and not a link.}

   \subsection{The Chain Complex} \label{subsec:ChainCx} We first describe the chain complex $(C^-(G_K), \bdry^-)$, defining gradings of its elements in the next subsection. 
   
   $C^-(G_K)$ is generated as a free module over the polynomial ring $\mathbb{Z}_2[U_0,\ldots, U_{n-1}]$ by a set $\mathcal{G}$ determined by $G_K$.  Elements of $\mathcal{G}$ consist of (unordered) $n$--tuples of intersection points in $\vec{\alpha} \cap \vec{\beta}$ which correspond to bijections between $\vec{\alpha}$ and $\vec{\beta}$.  We refer to the $n$ points comprising a generator ${\bf x} \in \mathcal{G}$ as the {\em components} of ${\bf x}$, and the unique component of $\x$ in $\alpha_i$ as the {\em $\alpha_i$--component}, denoted $x_i$.  By picking a fundamental domain for $G_K$ as in Figure~\ref{fig:twistedgrid} and ordering the  ${\alpha}$ curves in increasing order from bottom to top and the ${\beta}$ curves from left to right, we can identify generators $\x\in\mathcal{G}$ with elements in $S_n \times \mathbb{Z}_p^n$, where $S_n$ is the symmetric group on $n$ letters\footnote{Throughout this paper, we will represent an element $\sigma$ of $S_n$ by the ordered tuple $[\sigma(0) \,\, \dots \,\, \sigma(n-1)]$ of images of the elements $\{0, \ldots, n-1\}$.   See Figure~\ref{fig:Generator} for an example.}.  Indeed, to an element $\{\sigma,(a_0,\ldots, a_{n-1})\}\in S_n \times \mathbb{Z}_p^n$, we associate the unique ${\bf x} \in \mathcal{G}$ which satisfies:
   \begin{enumerate}
	   \item The $\alpha_i$ component of $\x$ lies in $\alpha_i\cap \beta_{\sigma(i)}$.
	   \item The $\alpha_i$ component of $\x$ is the $a_i$--th intersection between $\alpha_i\cap \beta_{\sigma(i)}$.  Here, the $p$ distinct intersection points of $\alpha_i\cap \beta_{\sigma(i)}$ are numbered in increasing order from $0$ to $p-1$ as $\alpha_i$ is traversed from left to right in this fundamental domain. 
   \end{enumerate}
It is clear that this correspondence is a bijection.

\begin{figure}
\begin{center}
\input{Figures/Generator-new.pstex_t}
\end{center}
\caption{A grid number $n=3$ diagram for a knot in the lens space L(5,2).  The black dots represent the generator specified by $\{[2\,\,0\,\,1], (4,2,3)\}$  in $S_3 \times (\mathbb{Z}_5)^3$.  The intersection points between a fixed $\alpha$ and $\beta$ curve are labeled $0, \ldots p-1$ when read from left to right on the preferred fundamental domain.}
\label{fig:Generator}
\end{figure}
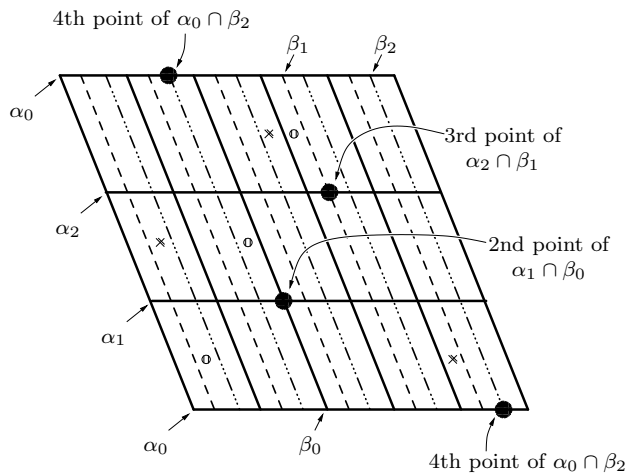

The boundary operator $\partial^-$ counts certain embedded parallelograms in $G_K$ which connect generators in $\mathcal{G}$. To describe it, let us call a  properly embedded quadrilateral in  $G_K$ a {\em parallelogram}.  Here, proper means that that alternating edges of the quadrilateral are identified with alternating subintervals of the $\vec{\alpha}$ and $\vec{\beta}$ curves, and vertices of the quadrilateral are identified with intersections $\alpha_i\cap \beta_j$.  See Figure~\ref{figure:ParGram}.  

We say that a parallelogram $\Par$ {\em connects $\x\in \mathcal{G}$ to $\y\in \mathcal{G}$}, if $\x$ and $\y$ agree for all but two components, $\{x_i,x_j\}$ and $\{y_i,y_j\}$, and the corners of $\Par$ are $\{x_i,y_i,x_j,y_j\}$, arranged so that the arcs on $\partial \Par$ along $\alpha_i$ (resp. $\alpha_j$) are oriented from $x_i$ to $y_i$ (resp. $x_j$ to $y_j$).  Here, $\partial P$ is oriented counter-clockwise with respect to the center of $\Par$. See Figure~\ref{figure:ParGram}.  

We call a parallelogram connecting $\x$ to $\y$ {\em admissible}  if it contains no components of $\x$ or $\y$ in its interior.  (The parallelogram shown in Figure~\ref{figure:ParGram} is admissible.)  For each $\x, \y \in \mathcal{G}$, form the set  
$$\PG({\bf x},{\bf y}) =  \{\Par | \ \Par \text{\ is an admissible\ parallelogram\ connecting\ } \x \text{\ to\ } \y \ \}.  $$
                      
The boundary operator is defined on generators ${\bf x}\in \mathcal{G}$ by 
$$\partial^-({\bf x}) = \sum_{{\bf y} \in \mathcal{G}}\sum_{\begin{subarray}{c}
                                                       \Par \in \PG({\bf x},{\bf y})\\
                                                       \end{subarray}} U_0^{n_{O_0}(\Par)}\cdots U_{n-1}^{n_{O_{n-1}}(\Par)} {\bf y}$$
where  $n_{{O_i}}(\Par)$ denotes the number of times $O_i$ appears in the interior of $\Par$.  We extend this to an operator on all of $C^-(G_K)$ by requiring linearity over addition, and equivariance with respect to each polynomial variable $U_i$.

\begin{proposition}\label{prop:Boundary}
$(C^-(G_K),\partial^-)$ is isomorphic to a chain complex which computes the knot Floer homology $(CF^-(L(p,q),K),\partial^-)$.  
\end{proposition}

Note that the identification implies that $(\partial^-)^2=0$, and a host of other properties satisfied by knot Floer homology chain complexes.  In particular, it implies that  $(C^-(G_K),\partial^-)$ is equipped with three gradings, the $\Spin^c$, Maslov, and Alexander gradings, denoted $\Sf$, $\Mf$, and $\Af$.  Moreover, it implies that $(C^-(G_K),\partial^-)$ is filtered with respect to $\Af$,  and that the filtered chain homotopy type of $(C^-(G_K),\partial^-)$ with respect to $\Af$ is an invariant of $K\subset L(p,q)$. We now discuss how to recover $\Sf$, $\Mf$, and $\Af$ from the combinatorics of $G_K$.

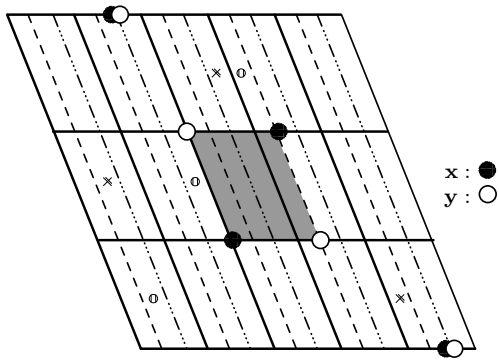
\begin{figure}
\begin{center}
\input{Figures/ParGram-new2.pstex_t}
\end{center}
\caption{Two generators in a grid number $3$ diagram for $L(5,2)$ connected by an admissible parallelogram connecting ${\bf x}$ and ${\bf y}$ with ${\bf x}$ at the NE-SW corners.}
\label{figure:ParGram}
\end{figure}

\subsection{Gradings} 
For a knot $K\subset L(p,q)$, any knot Floer homology chain complex $CF^-(L(p,q),K)$ comes equipped with three gradings:  the $\Spin^c$, Maslov, and rational Alexander gradings.  In light of Proposition~\ref{prop:Boundary}, we can hope to understand these gradings in terms of the combinatorics of $G_K$.   Indeed, we define three combinatorial quantities $$(\Sp,\M,\A)\in (\Z_p,\Q,\Q)$$ which will be identified with the aforementioned gradings on knot Floer homology.

\subsubsection{$\Sp$ grading} \label{subsubsec:SpinCGrad}
To describe the first grading, let ${\bf x}_{\OO}\in \mathcal{G}$ be the generator whose components  consist of the lower left corners of the $n$ distinct parallelogram regions in $T^2 - \vec{\alpha}-\vec{\beta}$ which contain the $\OO$ basepoints. Express ${\bf x}_{\OO}$ as an element of $S_n \times \Z_p^n$,  so that ${\bf x}_{\OO}= (\sigma_{\OO}, (a_0, \ldots, a_{n-1}))$.  Now let ${\bf x}=(\sigma,(b_0,\ldots, b_{n-1})) \in \mathcal{G}$ be any generator. Define

\[\Sp({\bf x}) = \left[ (q-1)+\left(\sum_{i=0}^{n-1} b_i - \sum_{i=0}^{n-1} a_i \right)\right] \mod p.\]

\noindent Now extend $\Sp$ to a grading on homogeneous elements in $C^-(G_K)$ by the rule $$\Sp(U_i{\bf x}) = \Sp({\bf x}) \ \ \ \ \mathrm{for \ each \ } i \in 0, \ldots, n-1.$$

The following proposition indicates that this grading corresponds to the $\Spin^c$ grading on knot Floer homology.  Before stating it, let us recall that knot Floer homology is equipped with a map $$\mathfrak{s}_\OO:CF^-(L(p,q),K)\rightarrow \Spin^c(L(p,q)),$$ where the term on the right is the set of $\Spin^c$ structures on $L(p,q)$. These, in turn are in affine isomorphism with $\Z_p$. 

\begin{proposition}\label{prop:SpincGrad} 
Under the identification between $C^-(G_K)$ and $CF^-(L(p,q),K)$ of Proposition~\ref{prop:Boundary}, we  have $$\Sp({\bf x}) = \Sf({\bf x}) := \phi\circ \s_\OO({\bf x}),$$
\noindent for all $\x$,  where $\phi$ is an explicit identification between $\Spin^c$ structures on $L(p,q)$ and $\Z_p$ given in Section~4.1 of \cite{MR1957829}, \end{proposition}

Note that the above identification indicates that $C^-(G_K)$ splits as a direct sum of $p$ subcomplexes, each freely generated as a $\Z_2[U_0,\ldots ,U_{n-1}]$ module by those $\x\in \mathcal{G}$ with a fixed value of $\Sp$.

\subsubsection{$\M$ grading}
Knot Floer homology is endowed with a grading, denoted $\gr$, which takes values in $\Q$ and was originally defined in terms of the Maslov index and characteristic classes.  In terms of the combinatorics of $G_K$, we can associate a rational number ${\bf M}({\bf x}) \in \Q$ to each generator ${\bf x}$ which we will show coincides  with $\gr$.   To do this,  first define a function 
\[W\colon \left\{\begin{array}{cc}\mbox{Finite sets}\\ \mbox{of points in } G_K\end{array}\right\} \rightarrow \left\{\begin{array}{cc}\mbox{Finite sets of pairs}\\ (a,b) \mbox{ with } a \in [0,pn), b \in [0,n)\end{array}\right\}\]
which assigns to each set of $n$ points in the preferred fundamental domain for $G_K$ (see Figure~\ref{fig:twistedgrid}) the $n$--tuple of its $\R^2$ coordinates, written with respect to the basis
\[ \left\{ \vec{v}_1 = \left(\frac{1}{np},0 \right), \vec{v}_2= \left(-\frac{q}{np},\frac{1}{n}\right) \right\} .\]

Let us require that $\OO$, $\XX$ lie in the centers of their respective parallelograms.  Then  $W({\bf x})$ has integer entries, while $W({\OO})$ and $W({\XX})$ have half-integer entries.
  
Next define a function 
\[C_{p,q}\colon \left\{\begin{array}{cc}\mbox{Finite sets of pairs}\\(a,b) \mbox{ with } a \in [0,pn),b \in [0,n)\end{array}\right\} \rightarrow \left\{\begin{array}{cc}\mbox{Finite sets of pairs}\\(a,b) \mbox{ with } a,b \in [0,pn)\end{array}\right\}\]
which sends an $n$--tuple of coordinates $$\{(a_i,b_i)\}_{i=0}^{n-1}$$ to the $pn$--tuple of coordinates $$\left((a_i + nqk)\mod np,\ b_i+nk\right)_{i=0,k=0}^{i=n-1,k=p-1}.$$

\noindent Let $\widetilde{W} = C_{p,q} \circ W$.  

Furthermore, let $\mathcal{I}$ be the function (defined in \cite{GT0610559}) whose input is an ordered pair $(A,B)$, where each of $A, B$ is a finite set of coordinate pairs.  $\mathcal{I}(A,B)$ is defined as the number of pairs $(a_1,a_2) \in A$, $(b_1,b_2) \in B$ such that $a_i < b_i$ for $i=1,2$.

$\M$ is now defined by:
$$\begin{array}{c}
{\bf M}({\bf x}) = \frac{1}{p}\left(\mathcal{I}(\widetilde{W}({\bf x}),\widetilde{W}({\bf x})) - \mathcal{I}(\widetilde{W}({\bf x}),\widetilde{W}(\OO)) - \mathcal{I}(\widetilde{W}(\OO),\widetilde{W}({\bf x})) + \mathcal{I}(\widetilde{W}(\OO),\widetilde{W}(\OO)) + 1\right) \\ + d(p,q,q-1) + \frac{p-1}{p},\\
\end{array}$$
\noindent  where $d(p,q,i)$ is the inductively-defined function:
\begin{align*}
 d(1,0,0) &= 0\\
 d(p,q,i) &= \left(\frac{pq - (2i+1-p-q)^2}{4pq}\right)-d(q,r,j),
\end{align*}
\noindent with $r$ and $j$ being the reductions modulo $q$ of $p$ and $i$, respectively.  
\begin{remark} \label{remark:orientation} See \cite{MR1957829} for an explanation of $d(p,q,i)$.  Here, we use the convention that $L(p,q)$ is $-\frac{p}{q}$ surgery on the unknot, which is opposite of the convention used in \cite{MR1957829}.  Their Heegaard diagram for $-L(p,q)$ is (according to our convention) a Heegaard diagram for $L(p,q)$.
\end{remark}

\noindent Extend ${\bf M}$ to the entire $\Z_2[U_0, \ldots, U_{n-1}]$ module by the rule
$${\bf M}(U_i{\bf x}) = {\bf M}({\bf x}) - 2 \mathrm{\ for\ each\ } i \in 0, \ldots, n-1.$$

The following proposition says that $\M$ agrees with the Maslov grading on Floer homology.
\begin{proposition}\label{prop:MaslovGrad}
	Under the identification between  $C^-(G_K)$ and $CF^-(L(p,q),K)$,  we have $$\M(\x)=\gr(\x),$$ for all $\x$. 
\end{proposition}
The identification of $\M$ with $\gr$ implies, in particular, that $\M(\partial^-(\x))=\M(\x)-1$, since $\Mf$ is the homological grading on Floer homology.


\subsubsection{$\A$ grading}\label{sec:combAlexGrad}
Knot Floer homology has another grading, $\Af\in \Q$, called the rational Alexander grading, which was originally defined in terms of Chern classes of relative $\Spin^c$ structures.  Moreover, $\Af$ equips $CF^-(L(p,q),K)$ with the structure of a filtered chain complex, in the sense that $\Af(\partial^-(\x))\le \Af(\x)$.  We can recover this grading as a combinatorial quantity, which we denote by $\A$.  Note that our definition of  $\M$ above depended on  the $n$--tuple of basepoints, $\OO$.  When we wish to emphasize this dependence, we write $\M_{\OO}$.  When using $\XX$ instead of $\OO$, we write $\M_{\XX}$.

Now we define
 $${{\A}}({\bf x}) = \frac{1}{2}\left({\M}_{\OO}({\x}) - {\M}_{\XX}({\x})- (n-1)\right),$$
\noindent which measures the difference in Maslov gradings associated to the different choices of basepoints. We extend ${\A}$ to a grading on the entire $\Z_2[U_0, \ldots, U_{n-1}]$ module by 
$${\A}(U_i{\x}) = {\A}({\x}) - 1 \mbox{ for each } i \in 0, \ldots, n-1.$$

This grading agrees with the rational Alexander grading on knot Floer homology defined using relative $\Spin^c$ structures on the knot complement:

\begin{proposition}\label{prop:AlexGrad} Under the identification between  $C^-(G_K)$ and $CF^-(L(p,q),K)$, we have $$\A(\x)=\Af(\x)$$ \noindent for all $\x$. 
\end{proposition}
In particular we see that, in light of the above comments, $\A(\partial^-(\x))\le \A(\x)$, and hence defines a filtration of $C^-(G_K)$.  It is the filtered chain homotopy type of  $C^-(G_K)$ which is the primary knot invariant associated to $(L(p,q),K)$.

\subsection{Algebraic Derivatives of $(C^-(G_K),\partial^-)$}

We quickly mention some relevant algebraic variations of this construction (cf.\ Section~2.3 of \cite{GT0610559}).

Propositions~\ref{prop:Boundary} and~\ref{prop:AlexGrad} indicate that $C^-(G_K)$ is a filtered complex.  As such, we can consider its associated graded complex, which we denote $(CK^-(G_K),\partial^-_K)$.   In particular, it has the same generators and gradings, but the differential is restricted to counting $\Par \in \PG({\bf x},{\bf y})$ which miss the $\mathbb{X}$ basepoints:  $$\partial^-_K({\bf x}) = \sum_{\substack{ \\y \in \mathbb{T}_{\alpha} \cap \mathbb{T}_\beta\\ }} \,\,\sum_{\substack{\Par \in \PG(x,y)\\  n_{\mathbb{X}}(\Par) = 0}} \,\, U_0^{n_{O_0}(\Par)}\cdots U_{n-1}^{n_{O_{n-1}}(\Par)} {\bf y}.$$

As another variant, we can consider the quotient complex of $C^-(G_K)$ obtained by setting $U_0 = 0$, with $\Sp,\M,\A$ gradings  induced from $C^-(G_K)$.  We this quotient by $\widehat{C}(G_K)$.  It is also filtered by $\A$, and we denote its associated graded complex by $\widehat{CK}(G_K)$.  The resulting homology groups, denoted $\widehat{HK}(G_K)$, are knot invariants and are isomorphic to the so-called knot Floer homology groups of $K$.   These are the groups usually denoted by $\widehat{HFK}(L(p,q),K)$.

\section{Background on Heegaard Floer homology and Rationally Null-homologous Knots} \label{section:HFBackground}
In order to prove the statements in Section~\ref{section:CombSummary}, we will need to recall some basic facts about knot Floer homology for knots in rational homology spheres ($\mathbb{Q}HS^3$'s).
The definitions in this section apply to all $\mathbb{Q}HS^3$'s, and not just lens spaces.  We refer the reader to \cite{MR2065507, GT0306378} for the original definitions of knot Floer homology, \cite{GT0512286, GT0607691, GT0610559} for discussions of the impact of extra basepoints, and \cite{GT0504404,GT0604360} for discussions of the structure of the invariant when $K$ is not null-homologous.

In the remainder of this section, let $Y$ be a $\mathbb{Q}HS^3$ and $K$ be an oriented knot in $Y$ (note that $K$ may not be null-homologous).  The starting point in the construction of a Heegaard Floer chain complex for $(Y,K)$ is a handlebody decomposition of $Y$ obtained from a particular Morse-Smale pair, $(f,h)$, where $f$ is a self-indexing Morse function $f\colon Y \rightarrow \mathbb{R}$ with $n$ index $3$ and $n$ index $0$ critical points and $h$ is a Riemannian metric. 

One chooses $(f,h)$ so that $K$ can be realized as the union $\gamma_\XX \cup \gamma_\OO$ of $n$ (upward) flowlines $\gamma_{\mathbb{X}}$ of $\nabla(f)$ which join the index $0$ critical points bijectively to the index $3$ critical points and $n$ (downward) flowlines $\gamma_{\mathbb{O}}$ of $-\nabla(f)$ which join the index $3$ critical points bijectively to the index $0$ critical points.  The level set $\Sigma = f^{-1}(\frac{3}{2})$ is a surface of genus $g$ which provides a Heegaard splitting surface for the handlebody decomposition associated to $f$.  In particular, $\Sigma$ divides $Y$ into two genus $g$ handlebodies $Y_\alpha = f^{-1}[0,\frac{3}{2}]$ and $Y_\beta = f^{-1}[\frac{3}{2},3]$, with $\Sigma = \partial Y_\alpha = -\partial Y_\beta$.

\begin{definition} \label{def:CompHeegDiag} Such a pair, $(f,h)$, determines a five-tuple $(\Sigma, \vec{\alpha}, \vec{\beta}, \OO, \XX)$ called a {\it Heegaard diagram for Y compatible with the knot K}. Here, 
\begin{itemize}
\item  $\vec{\alpha} = \{\alpha_0, \ldots, \alpha_{g+n-2}\}$ are the $g+n-1$  simple closed curves that arise as the intersection of $\Sigma$ with the stable submanifolds (with respect to $-\nabla(f)$) of the index $1$ critical points.   They are mutually disjoint and span a $g$--dimensional subspace  $V\subset H_1(\Sigma;\R)$.

\item $\vec{\beta} = \{\beta_0, \ldots, \beta_{g+n-2}\}$ are the $g+n-1$  simple closed curves that arise as the intersection of $\Sigma$ with the unstable submanifolds (with respect to $-\nabla(f)$) of the index $2$ critical points.   They are mutually disjoint and span a $g$--dimensional subspace  $V'\subset H_1(\Sigma;\R)$, complementary to $V$.

\item $\OO = \{O_0, \ldots O_{n-1}\}$  are $n$ points in $\Sigma - \vec{\alpha} - \vec{\beta}$ which arise as the intersection of $\Sigma$ with the $n$ flowlines in $\gamma_\OO$.
\item $\XX = \{X_0, \ldots, X_{n-1}\}$ are $n$ points in $\Sigma - \vec{\alpha} - \vec{\beta}$ which arise as the intersection of $\Sigma$ with the $n$ flowlines in $\gamma_\XX$.
\end{itemize}
\end{definition}

Fix an oriented meridian $\mu$ of the oriented knot $K$ as the boundary of a small disk neighborhood of a point of $\OO$ on $\Sigma$ as indicated in Figure~\ref{fig:meridian}.  This meridian provides fixed generators $[\mu]$ for $H_1(Y-K;\Q)$ and $\PD[\mu]$ for $H^2(Y,K;\Q)$.

\begin{figure}
\centering
\input{Figures/Meridian.pstex_t}
\caption{A local picture of $K$ at a point of $\OO$ on $\Sigma$ and the corresponding fixed meridian $\mu$.}
\label{fig:meridian}
\end{figure}
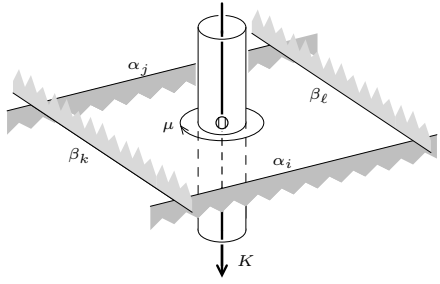

Let $\Sym^{g+n-1}(\Sigma)$ denote the $(g+n-1)$-fold symmetric product of $\Sigma$, and let $\mathbb{T}_{\alpha}$ and $\mathbb{T}_{\beta}$ denote the image of $\alpha_0 \times \ldots \times \alpha_{g+n-2}$ and $\beta_0 \times \ldots \times \beta_{g+n-2}$ contained within it.  The chain complex $(CF^-(Y,K), \bdry^-)$ is freely generated as a module over the polynomial ring $\mathbb{Z}_2[U_0, \ldots, U_{n-1}]$ by the elements $\x \in \T_{\alpha} \cap \T_{\beta}$.   The boundary map is defined on generators by
\[\partial^-\x = \sum_{\substack{ \\{\bf y} \in \T_{\alpha} \cap \T_\beta\\ }} \,\,\sum_{\substack{\phi \in \pi_2(\x,{\bf y})\\ \mu(\phi) = 1}} \# \widehat{\mathcal{M}}(\phi) \,\, U_0^{n_{O_0}(\phi)}\cdots U_{n-1}^{n_{O_{n-1}}(\phi)} {\bf y}.\]
and extends to the entire module by linearity over addition and equivariance with respect to the action of each variable, $U_i$.   In the double summation,  $\pi_2({\bf x},{\bf y})$ denotes the set of homotopy classes of maps of Whitney disks connecting ${\bf x}$ to ${\bf y}$, $\# \widehat{\mathcal{M}}(\phi)$  the signed number of points, modulo $2$, in the moduli space of unparameterized holomorphic representatives of $\phi$ (with respect to a fixed, generic family of almost complex structures on $\Sym^{g+n-1}(\Sigma)$),  $\mu(\phi)$  the Maslov index (expected dimension of the moduli space) of $\phi$, and $n_{O_i}(\phi)$ the algebraic intersection number of $\phi$ with 
\[\{O_i\} \times \Sym^{g+n-2}(\Sigma).\]

\noindent See \cite{MR2113019} and \cite{MR2065507} for further details.

\begin{remark} Strictly speaking, this definition of $CF^-(Y,K)$ is different from the definition in \cite{MR2065507}, since it allows for the use of multiple basepoints.  However, the fact that the filtered chain homotopy type of $CF^-(Y,K)$ over the ring $\mathbb{Z}_2[U_0, \ldots, U_{n-1}]$ agrees with the definition given above follows easily from the proof of the corresponding statement when $Y=S^3$, Proposition~2.3 of \cite{GT0607691}.  Indeed, any multiply-pointed Heegaard diagram for a knot in a $\Q HS^3$ can be reduced to the special form of Lemma~2.4 of \cite{GT0607691}, and the necessary gluing theorem {\em (}Proposition~6.5 of \cite{GT0512286}{\em)} holds in this context.
\end{remark}

\subsection{$\Spin^c$ structures, Maslov, and Alexander gradings} 
\label{subsec:GradDef}
The elements of $CF^-(Y,K)$ are endowed with three gradings.
Although we have primarily restricted attention to knots, we relax this restriction when discussing gradings.  We do this because the lift of a knot to a covering manifold may be a link, and it will be necessary to understand how to assign gradings in that case for the proofs of Propositions~\ref{prop:MaslovGrad} and \ref{prop:AlexGrad}. 
Thus, until Lemma~\ref{lemma:Symmetry} at the end of this section, let us assume $K$ is link of $\ell$ components.  The $i$th component $K_i$ of $K$ accounts for $n_i$ of the $\OO$ basepoints; $n_1 + \cdots + n_\ell = n$.  Let $\mu_i$ be an oriented meridian for $K_i$ as depicted in Figure~\ref{fig:meridian}.

Associated to each generator $\x \in \T_\alpha \cap \T_\beta$ is 
\begin{itemize}
  \item a {\em homological {\em(}Maslov{\em )} grading} $\Mf(\x) \in \Q$, and
  \item a {\em relative $Spin^c$ structure} $\us_{\OO,\XX}(\x) \in \underline{\Spin}^c(Y,K)$, which naturally gives rise to a pair $(\s_\OO(\x),\Af_{\OO,\XX}(\x)) \in \Spin^c(Y) \times \Q$.  The first term in the pair is the {\em $Spin^c$ grading} and the second is the {\em filtration} or {\em Alexander grading}.
\end{itemize}
These three gradings extend to gradings on the entire chain complex by:
\begin{align*}
\s_\OO(U_i\x) &= \s_\OO(\x),\\
\Af_{\OO,\XX}(U_i\x) &= \Af_{\OO,\XX}(\x) - 1, \mbox{ and} \\
\Mf(U_i\x) &= \Mf(\x) - 2
\end{align*}
for all $i \in 0, \ldots, n-1$.

The {\it Maslov grading} $\wtgr({\bf x})$ associated to a generator ${\bf x} \in \mathbb{T}_{\alpha} \cap \mathbb{T}_{\beta}$  is characterized (see Theorem~7.1 in \cite{MR2222356})\footnote{The condition that ${\bf x}$ be associated to a torsion $\Spin^c$ structure is automatically satisfied, since $Y$ is a $\Q HS^3$.} by the following properties:

\begin{enumerate}
\item ${\wtgr}({\xi}) = 0$ for $\xi$ the homogeneous generator of ${HF}^-(S^3) \cong \Z_2[U]$,\\
\item If ${\bf x},{\bf y} \in \mathbb{T}_\alpha \cap \mathbb{T}_\beta$, and $\phi \in \pi_2({\bf x},{\bf y})$, then $${\wtgr}({\bf x}) - {\wtgr}({\bf y}) = \mu(\phi) - 2n_{\OO}(\phi),$$ where $\mu(\phi)$ is the Maslov index (expected dimension of the moduli space of holomorphic representatives) of $\phi$ and $n_{\OO}(\phi) = \sum_{i=1}^n n_{O_i}(\phi)$.  In particular, the differential lowers ${\wtgr}$ by $1$,\\
\item Let $(\Sigma, \vec{\alpha},\vec{\gamma},\vec{\beta},\OO)$ be a Heegaard triple diagram associated to a presentation of $Y_{\alpha\beta}$ (the $3$--manifold whose Heegaard diagram is given by $(\Sigma, \vec{\alpha},\vec{\beta}, \OO)$) as surgery on a link in $Y_{\alpha\gamma} = S^3$ (as in Definition~4.2 of \cite{MR2222356}).  Then if ${\bf x} \in \mathbb{T}_{\alpha} \cap \mathbb{T}_{\gamma}$, ${\bf y} \in \mathbb{T}_\alpha \cap \mathbb{T}_\beta$, and $\widehat{\theta}$ the canonical top-degree generator of $HF^-(Y_{\gamma\beta})$ ($Y_{\gamma\beta}$ is a connected-sum of  $S^1 \times S^2$'s, cf.\ the discussion in the proof of Proposition~4.3 in \cite{MR2222356}), then $$\wtgr({\bf y}) - \wtgr({\bf x}) = -\mu(\psi) + 2n_{\OO}(\psi) + \frac{c_1(\t)^2 - 2\chi(W) - 3\sigma(W)}{4}.$$  Here $W$ is the cobordism associated to surgery on the link, $\t$ is a $\Spin^c$ structure on $W$ which restricts to the $\Spin^c$ structures $\s_\OO({\bf x}), \s_\OO({\bf y})$ on $Y_{\alpha\gamma}, Y_{\alpha\beta}$, respectively, and $\psi$ is a Whitney triangle in $Sym^{g+n-1}(\Sigma)$ connecting ${\bf x},\widehat{\theta},{\bf y}$, as in Section~8.1 of \cite{MR2113019}.\footnote{The Heegaard triple-diagram must, furthermore, be strongly $\t$--admissible, in the sense of Definition~8.8 of \cite{MR2113019}.}
\end{enumerate}

To understand the other two gradings, we must first recall some definitions.  First, let us say that two vector fields on a smooth manifold $M$ are {\em homologous} if they are homotopic in the complement of a finite number of open balls.  Now let us define a {\em $\Spin^c$ structure} on a closed, oriented, three-manifold, $M$, to be a homology class of nowhere-vanishing vector fields on $M$.  For an oriented three-manifold $M$ whose boundary consists of a collection of tori, a {\em relative $\Spin^c$ structure  on $M$ } is a homology class of nowhere-vanishing vector fields on $M$ pointing outward along $\partial M$ (here, the homotopy between homologous vector fields must be through fields pointing outward).  For $M = Y-N(K)$, the set of relative $\Spin^c$ structures is denoted by $\underline{\Spin}^c(Y,K)$.  We refer the reader to \cite{MR1484699} and also Section~I.4 of \cite{MR1958479} for more details.

Let us begin by describing the $\Spin^c$  and relative $\Spin^c$ structures associated to $\x\in \T_\alpha \cap \T_\beta$.  We will follow Section~3.6 of \cite{GT0512286} to construct nowhere-vanishing vector fields on $Y$ and $Y-N(K)$ associated to a generator $\x$ and the basepoints.

Begin with the vector field $\nabla(f)$ on $Y$.  Each component $x_i$ of the  generator $\x$ is the intersection of a gradient flowline, $\gamma_{x_i}$, connecting an index $1$ critical point to an index $2$ critical point of $f$.    To remedy the vanishing of $\nabla(f)$ at the index $1$ and $2$ critical points, choose any nowhere-vanishing extension of $\nabla(f)$ to neighborhoods of the $\gamma_{x_i}$.  To remedy the vanishing at the index $0$ and $3$ critical points, alter the neighborhoods of the components of $\gamma_{\OO}$ as in Figure~2 of \cite{GT0512286} so that $K$ is a collection of (oriented) closed orbits of the resulting vector field.  Define the homology class of this vector field to be $\s_\OO(\x)$.  This defines a map
\[ \s_\OO \colon \T_\alpha \cap \T_\beta \to \Spin^c(Y).\]

The above vector field contains $K$ as a collection of closed orbits and, moreover, the induced vector field on $Y-N(K)$ has a standard nowhere-vanishing vector field on its boundary.  One easily constructs a homotopy in a collar neighborhood of $\partial(Y-N(K))$ to a vector field oriented outward along $\partial(Y-N(K))$.  Let the homology class of this resulting vector field be denoted $\us_{\OO,\XX}(\x)$, so that we obtain a map
\[\us_{\OO,\XX} \colon \T_\alpha \cap \T_\beta \to \underline{\Spin}^c(Y,K).\]
(See also Section~2.3 of \cite{MR2065507} and Section~2.2 of \cite{GT0504404}.)  

Given a generator $\x \in \T_\alpha \cap \T_\beta$, its {\em relative $\Spin^c$ grading} is defined to be $\us_{\OO,\XX}(\x)$.  From $\uSpin^c(Y,K)$ there are two natural maps 
 $$\pi_{\mathfrak{s}}\colon \underline{\Spin}^c(Y,K) \rightarrow \Spin^c(Y)$$
  and 
  $$\pi_\A\colon \underline{\Spin}^c(Y,K) \rightarrow \Q^\ell$$
from which we obtain the $\Spin^c$ grading and the Alexander grading of $\x$.

The first map $\pi_\s$ is obtained by reversing the procedure described above that produces a vector field in $\underline{\Spin}^c(Y,K)$ from a vector field in $\Spin^c(Y)$.  Indeed, we have  $\s_\OO({\bf x}) = \pi_{\s} \circ \us_{\OO,\XX}({\bf x})$, and this composition defines the {\em $\Spin^c$ grading} of $\x$.

The second map $\pi_\A$ is obtained as in Section~4.4 of \cite{GT0604360} and measures the rational linking of the various components $K$ with the Poincar{\'e} dual of a certain cohomology class associated to a relative $\Spin^c$ structure.   
Recall we are assuming $K$ is a link of $\ell$ components $K_i$, $n_i$ is the number of $\OO$ basepoints for $K_i$, and $\mu_i$ is an oriented meridian for $K_i$.
Now define
\[\mathfrak{H}_\OO \colon \underline{\Spin}^c(Y,K) \to H^2(Y,K;\Q)\]
where
\[\mathfrak{H}_\OO(\us) = \frac{c_1(\us) - \sum_{i=1}^{\ell}(2n_i - 1)\PD[\mu_i]}{2}.\]

The map 
\[c_1 \colon \underline{\Spin}^c(Y,K) \to H^2(Y,K;\Q)\] 
is given by 
\[ c_1(\us) = \us - J\us.\] 
Here the involution 
\[J \colon \underline{\Spin}^c(Y,K) \to \underline{\Spin}^c(Y,K)\]
 is defined by first taking a representative vector field, $v$, to its reverse, $-v$,  and then performing a homotopy of $-v$ in a collar neighborhood of the boundary tori so that the resulting field points outward.

For each $i = 1, \dots, \ell$, let $h_i \in H_2(Y,K;\Q)$ denote the hom dual of $\PD[\mu_i]$ and define $\pi_{A_i}$ by  
\[\pi_{A_i}(\s) = \langle\mathfrak{H}_\OO(\us),h_i\rangle \in \Q.\]
$\pi_{\A}(\s)$ is then defined as 
\[\pi_{\A}(\s) = (\pi_{A_1}(\s) ,...,\pi_{A_\ell}(\s) )\in \Q^\ell.\]

From this we can define the {\em Alexander multi-grading}, of a generator $\x \in \T_\alpha \cap \T_\beta$ to be the $\ell$-tuple of rational numbers
$$( A^\mathfrak{f}_1(\x), ... A^\mathfrak{f}_l(\x))\in \Q^l,$$
\noindent where 
$$A^\mathfrak{f}_i(\x)= \pi_{A_i} \circ \us_{\OO,\XX}(\x)  =\langle\mathfrak{H}_\OO \circ  \us_{\OO,\XX}({\bf x}),h_i\rangle \in \Q.$$

 The {\em Alexander grading}, $\Af_{\OO,\XX}$, of a generator $\x \in \T_\alpha \cap \T_\beta$ is then defined to be the sum of the components in the multi-grading
\[\Af_{\OO,\XX}({\x})=  \sum_{i=1}^\ell A^\mathfrak{f}_i(\x). \]

This formula for the Alexander grading is a generalization of the formula given in \cite{GT0512286} and \cite{GT0604360}, normalized so that the homology satisfies certain symmetries (as in Section~8 of \cite{GT0512286} and Section~4.3 of \cite{GT0604360}).  See Lemma~\ref{lemma:Symmetry} below.

In this way, we obtain the three gradings, $\Mf_{\OO}({\bf x}), \s_{\OO}({\bf x}), \Af_{\OO,\XX}({\bf x})$, associated to $\x \in \mathbb{T}_\alpha \cap \mathbb{T}_\beta$.  Note that the Maslov and $\Spin^c$ gradings depend only on the $\OO$ basepoints, while the Alexander grading depends on both the $\OO$ and $\XX$ basepoints.

The chain complex $CF^-(Y,K)$ splits according to elements of $\Spin^c(Y)$:  
\[CF^-(Y,K) = \bigoplus_{\mathfrak{s} \in Spin^c(Y)} CF^-(Y,K;\mathfrak{s}).\]
For a fixed $\s \in \Spin^c(Y)$, the rational $\Af_{\OO,\XX}$ and $\Mf$ gradings are lifts of a relative $\Z$ grading.  That is, for ${\bf x},{\bf y} \in \mathbb{T}_\alpha \cap \mathbb{T}_\beta$ satisfying $\mathfrak{s}_{\mathbb{O}}({\bf x}) = \mathfrak{s}_{\mathbb{O}}({\bf y})$, we have that 
\begin{align*}
\Af_{\OO,\XX}(\x) - \Af_{\OO,\XX}(\y) & \in \Z, \mbox{ and} \\
\Mf_\OO(\x) - \Mf_\OO(\y) & \in \Z.
\end{align*}

Exchanging the roles of $\OO$ and $\XX$ has the effect of reversing the orientation on $K$.  This induces the $2n$--pointed Heegaard diagram $(\Sigma, \vec{\alpha},\vec{\beta}, \XX,\OO)$ compatible with the pair $(Y,-K)$.  With the appropriate exchanges, we obtain the gradings $\s_\XX(\x)$, $\us_{\XX,\OO}(\x)$, $\Af_{\XX,\OO}(\x)$, and $\Mf_\XX(\x)$ for a generator $\x \in \T_\alpha \cap \T_\beta$.  For the definition of $\Af_{\XX,\OO}(\x)$, note that whereas $\mu_i$ is the oriented meridian of $K_i$, its reverse $-\mu_i$ is the oriented meridian of $-K_i$.  

We close the section with a lemma about a symmetry of the Alexander grading under orientation reversal of $K$ in the case that $K$ is a knot.  This will be useful in proving that the combinatorial definition of the $\A$ grading given in Section~\ref{section:CombSummary} matches the definition of the $\Af$ grading detailed above.

\begin{lemma}\label{lemma:Symmetry}
Let $K$ be an oriented knot in a rational homology sphere $Y$.   Let $\x\in\T_{\alpha} \cap \T_\beta$ be a generator in $CF^-(Y,K)$, associated to a $2n$--pointed Heegaard diagram $(\Sigma, \vec{\alpha}, \vec{\beta}, {\OO},{\XX})$ for the pair $(Y,K)$.  Then
 \[\Af_{\XX,\OO}(\x) = -\Af_{\OO,\XX}(\x) - (n-1).\]
\end{lemma}

\begin{proof}
Let $\mu$ be the oriented meridian for $K$, and $h$ the hom dual to $\PD[\mu]$, as before.  With the reversed orientation, $-K$ has oriented meridian $-\mu$ and $-h$ is hom dual to $\PD[-\mu]$.  

Accounting for multiple basepoints, the argument used to prove  the second half of Lemma~3.12 in \cite{GT0512286}  yields
\[\us_{\XX,\OO}(\x) = \us_{\OO,\XX}(\x) - n\PD[\mu] \]
 for all $\x \in \mathbb{T}_{\alpha} \cap \mathbb{T}_\beta$.  
 Then since 
 \[c_1\left(\us_{\OO,\XX}(\x) - n \PD[\mu]\right) = c_1(\us_{\OO,\XX}({\bf x})) - 2n \PD[\mu]\]
 we have
\begin{align*}
\Af_{\XX,\OO} 
	& = \left\langle \frac{c_1(\us_{\XX,\OO}(\x))- (2n-1)\PD[-\mu]}{2} , -h \right\rangle\\
	&= \left\langle \frac{c_1\left(\us_{\OO,\XX}(\x) - n \PD[\mu]\right)+ (2n-1)\PD[\mu]}{2} , -h \right\rangle\\
	&= \left\langle \frac{c_1(\us_{\OO,\XX}(\x)) -  2n \PD[\mu]+(2n-1)\PD[\mu]}{2} , -h \right\rangle\\
	&=- \left\langle \frac{\left(c_1(\us_{\OO,\XX}(\x)) -  (2n-1) \PD[\mu]\right)+ (2n-2)\PD[\mu]}{2} , h \right\rangle\\
	&=- \Af_{\OO,\XX}(\x) - (n-1)
\end{align*}
as desired.
\end{proof}

\section{Proof of Theorem~\ref{thm:main} }\label{section:GridDiagrams}
In this section we prove Theorem~\ref{thm:main}.  We begin by showing that any link $K$ (and hence any knot) in a lens space $L(p,q)$ admits a grid diagram $G_K$ and observe that a grid diagram $G_K$ for the pair $(L(p,q),K)$ is actually a multiply pointed Heegaard diagram compatible with $K$.

As such, a filtered chain complex $CF^-(L(p,q),K)$ for a knot $K$ in a lens space $L(p,q)$ is associated to the grid diagram by the Heegaard Floer machinery described in Section~\ref{section:HFBackground}.  Proposition~\ref{prop:Boundary} shows that  $(CF^-(L(p,q),K),\partial^-)$ is isomorphic to the complex $(C^-(G_K),\partial^-)$ described in Subsection~\ref{subsec:ChainCx} using arguments analogous to those in  \cite{GT0607691}  (i.e.\ the Heegaard diagram is ``nice'' in the sense of  \cite{GT0607777}).  The remainder of the section is devoted to  proving that  the combinatorial quantities $(\Sp,\M,\A)$ associated to $C^-(G_K)$ agree with the $\Spin^c$ ($\Sf$), homological ($\Mf$), and Alexander ($\Af$) gradings, respectively, on knot Floer homology.

\subsection{Grid diagrams}
Here we show that any link $K\subset L(p,q)$ possesses a grid diagram in the sense of Definition~\ref{def:TTGridDiag}.  To begin, let $V_\alpha \cup_\Sigma V_\beta$ be a genus $1$ Heegaard splitting of the lens space $L(p,q)$.   There is a height function $h \colon L(p,q) \to [-\infty, +\infty]$ for which
\begin{itemize}
\item $h^{-1}\{0\} = \Sigma$,
\item $h^{-1}[-\infty,0] = V_\alpha$ and $h^{-1}[0,+\infty] = V_\beta$, and 
\item $h^{-1}\{\pm \infty\}$ are the core curves of $V_\alpha$ and $V_\beta$, respectively.
\end{itemize}
We shall regard the solid torus $V_\alpha$ as lying below the Heegaard torus $\Sigma$ and $V_\beta$ as above. 

Let $K$ be a link in $L(p,q)$.  If $K$ is disjoint from $h^{-1}\{\pm \infty\}$, then $K$ is contained in $h^{-1}(-\infty,+\infty)$ which may be identified with $\Sigma \times (-\infty,+\infty)$ (where we equate the torus $\Sigma$ with $\Sigma \times \{0\}$).  If furthermore under the projection $\pi \colon \Sigma \times (-\infty,+\infty) \to \Sigma \times \{0\}$ the image of $K$ has at worst finitely many transverse double points, then we say $K$ is in {\em general position}.  The image $\pi(K)$ on $\Sigma$ together with over/under markings for the arcs through each double point is sufficient information to reconstruct $K$ up to isotopy; this information is known as a {\em diagram} of $K$.  By a slight isotopy, any knot or link may be put into general position and thus has a diagram.

As before, let $T^2$ be the oriented torus viewed as the quotient $\R^2 /\Z^2$ where $\Z^2$ is the standard lattice generated by the vectors $(1,0)$ and $(0,1)$.  Rule $T^2$ with {\em horizontal circles} of slope $0$ inherited from the horizontal lines in $\R^2$ and {\em slanted circles} of slope $-\frac{p}{q}$ inherited from the lines of slope $-\frac{p}{q}$ in $\R^2$.  We shall refer to $T^2$ as the {\em standard torus}.

Identify $\Sigma$ with $T^2$ so that the horizontal circles are meridians of $V_\alpha$ and the slanted circles are meridians of $V_\beta$.  Further identify these
 two solid tori each with a copy of $S^1 \times D^2$ so that the boundaries of the meridional disks $\{\mbox{pt}\} \times D^2$ are the horizontal and slanted circles  accordingly.
 
 \begin{definition}
 A link $K$ in a lens space $L(p,q)$ is in {\em grid position} if
(i) each component of $K$ is comprised of arcs, each properly embedded in a meridional disk $\{\mbox{pt}\} \times D^2$ of alternately $V_\alpha$ and $V_\beta$, and (ii) no two arcs of $K$ are contained in the same meridional disk.  If $K$ is decomposed in this way into $n$ arcs in $V_\alpha$ and $n$ arcs in $V_\beta$, then $n$ is the {\em grid number}.
\end{definition}

We shall say an immersed $1$--manifold in $\Sigma$ is {\em rectilinear} if
it is composed of a finite number of alternately horizontal and slanted
segments and no singularity occurs at an end point of a segment.  A {\em
segment} is an arc of a horizontal or slanted circle.  Here, and in the
definition of grid position, we permit degenerate arcs that are just
single points and those whose two end points coincide, forming an entire
circle.  We relax the definition of rectilinear so that a singularity of
the immersed $1$--manifold may occur at the end point of a degenerate
segment.

It is clear, then, that any link $K \subset L(p,q)$ in grid position gives rise to a twisted toroidal grid diagram $G_K = (T^2, \vec{\alpha}, \vec{\beta}, \OO, \XX)$ for $K$ in the sense of Definition~\ref{def:TTGridDiag} and vice versa.  In fact, via the identification of the Heegaard torus $\Sigma$ with $T^2$, $G_K$ is a $2n$--basepointed Heegaard diagram compatible with $K$, in the sense of Definition~\ref{def:CompHeegDiag}.

\begin{lemma}\label{lem:rectilinear}
Any diagram $D(K)$ of a link $K$ in $L(p,q)$ is isotopic to one that is rectilinear with respect to the ruling, and such that every undercrossing is horizontal.
\end{lemma}
\begin{proof}
Arrange by an isotopy of the diagram $D(K)$ that in a small neighborhood $N_c$ of each crossing $c$, the under arc is horizontal and the over arc is slanted.  Outside these neighborhoods, the diagram is a collection of disjoint embedded arcs $\mathcal{A} = D(K) - \bigcup \Int N_c$.  Each arc $a \in \mathcal{A}$ is arbitrarily close ---and hence isotopic rel--$\bdry$--- to a finite polygonal arc $a'$ with the same endpoints which is composed of alternately horizontal and slanted segments.
Moreover, we may assume that each such arc $a'$ meets the horizontal arc of an under crossing with a slanted arc and the slanted arc of an over crossing with a horizontal arc.  After isotoping these arcs $a \in \mathcal{A}$ to their rectilinear approximations $a'$, $D(K)$ is rectilinear and every under crossing is horizontal.
\end{proof}

\begin{proposition} \label{prop:existence}
Each link $K \subset L(p,q)$ is described by a grid diagram.
\end{proposition}

\begin{proof}
Let $K$ be a link in the lens space $L(p,q)$.  Let $D(K)$ be a diagram of $K$ on a Heegaard torus $\Sigma$.  By the above Lemma~\ref{lem:rectilinear}, we may assume $D(K)$ is rectilinear and every under crossing is horizontal.
If two or more horizontal segments of $D(K)$ occur on the same horizontal circle, then within a small open neighborhood of this circle meeting no other horizontal segments we may isotope each horizontal segment along the slanted ruling to lie within its own horizontal circle.  This does not effect which slanted loops the slanted segments lie upon, though it may alter their lengths.  Similarly, we can arrange that no two slanted segments of $D(K)$ lie on the same loop.
Having performed such isotopies as needed, $D(K)$ is now a grid diagram for $K$.
\end{proof}

\subsection{Identification of $(CF^-({G_K}),\partial^-)$ with $(CF^-(L(p,q),K),\partial^-)$ }
We now identify the combinatorial chain complex for a knot $K$ in a lens space $L(p,q)$ described in Subsection~\ref{subsec:ChainCx} with the knot Floer homology chain complex associated to the Heegaard diagram of $G_K$.
\begin{proof}[Proof of Proposition~\ref{prop:Boundary}]
As observed in the preceding subsection, a grid diagram for $K\subset L(p,q)$ is actually a compatible, $2n$--pointed Heegaard diagram for $K$, and hence defines a chain complex  $(CF^-(L(p,q),K),\partial^-)$.  In fact, one sees immediately that this Heegaard diagram is admissible (a technical requirement for Heegaard Floer homology of $\Q HS^3$'s with multiple basepoints) by an argument analogous to that given in the paragraph following Definition 2.2 in \cite{GT0607691}.  The identification between $\x\in \mathcal{G}$ and $\x\in \mathbb{T}_{\alpha}\cap \mathbb{T}_{\beta}$ is immediate, and hence the underlying (ungraded) $\Z_2[U_0,\ldots U_{n-1}]$ modules associated to $CF^-({G_K})$ and $CF^-(L(p,q),K)$ are isomorphic.

 Moreover, the Heegaard diagram is ``nice'' in the sense of \cite{GT0607777}.   In fact, every region of $T^2-\vec{\alpha}-\vec{\beta}$ is a quadrilateral, and thus the full chain complex $(CF^-(L(p,q),K),\partial^-)$  is combinatorially computable.  According to Theorem~3.2 of \cite{GT0607777}, the boundary operator  counts only Whitney disks whose domains are embedded $4$--sided  polygons (i.e.\ parallelograms) of Maslov index $1$. The condition on the Maslov index forces the parallelograms to be admissible, showing that the boundary operator for $(CF^-(L(p,q),K),\partial^-)$ is given by

$$\partial^-({\bf x}) = \sum_{{\bf y} \in \mathcal{G}}\sum_{\begin{subarray}{c}
                                                       \phi \in \PG({\bf x},{\bf y})\\
                                                      n_{\bf x}(\phi) = n_{\bf y}(\phi) = 0\\                                
						       \end{subarray}} U_0^{n_{O_0}(\phi)} \cdots U_{n-1}^{n_{O_{n-1}}(\phi)} {\bf y},$$

\noindent as desired.  
\end{proof}

\subsection{Identification of the $\Spin^c$, Maslov, and Alexander gradings with $\Sp,\M,\A$}
We now show that the combinatorial quantities $\Sp,\M,$ and $\A$ defined on $C^-(G_K)$ agree with the gradings it inherits as a chain complex for the knot Floer homology of the knot $K\subset L(p,q)$. 

\subsubsection{$\Spin^c$ gradings and the Proof of Proposition~\ref{prop:SpincGrad}.}
We first handle the $\Spin^c$ grading.  As noted in Section~\ref{section:HFBackground}, the filtered chain complex for $CF^-(Y,K)$ for a knot $K$ in a rational homology sphere $Y$ splits according to elements of $\Spin^c(Y)$, which is an affine set for an action of $H_1(Y;\mathbb{Z})\cong \Z_p$.

In \cite{MR1957829}, \ons construct an affine identification\footnote{Both sets admit an action by $H_1(L(p,q);\Z)$ which, in the case of $\Z_p$, comes from an implicit isomorphism $H_1(Y;\mathbb{Z})\cong \Z_p$ induced from the Heegaard diagram.  Note that while this identification at first sight appears to be solely in terms of a specific Heegaard diagram for $L(p,q)$, it has a geometric interpretation in terms of the Chern classes of $\Spin^c$ structures over four-dimensional two-handle cobordisms between lens spaces.}
$$ \phi: \Spin^c(L(p,q);\Z)\rightarrow \Z_p.$$

For the standard singly-pointed genus $1$ Heegaard diagram for $L(p,q)$, we have $$\phi(\mathfrak{s}_\OO(\x_\OO)) =q-1, $$ where $\mathfrak{s}_\OO(\x_\OO)$ is the $\Spin^c$ structure corresponding to the intersection point located in the lower left-hand corner of the region containing $\OO$ (there, $w$).  

Recall from Subsection~\ref{subsec:GradDef}, the map 
$$\mathfrak{s}_{\mathbb{O}}: \mathbb{T}_{\alpha} \cap \mathbb{T}_{\beta}=\mathcal{G} \rightarrow \Spin^c(Y).$$
Composing $\mathfrak{s}_{\mathbb{O}}$ with $\phi$, and considering the Heegaard diagram associated to a grid diagram $G_K$, we obtain a map:
 $$\Sf=\phi \circ \mathfrak{s}_{\OO}\colon \mathcal{G} \rightarrow \Z_p.$$ 

We wish to show that $\Sf(\x)=\Sp(\x)$ for all $\x\in \mathcal{G}$, where $\Sp$ is the combinatorial quantity defined in Subsection~\ref{subsubsec:SpinCGrad}.  The first step is to show that they agree for a specific element ${\bf x}_{\OO}\in \mathcal{G}$.

\begin{lemma} \label{lemma:SpincCanGen}
Let ${\bf x}_{\OO}$ be the generator whose components lie in the lower left hand corners of the regions in $G_K$ containing the $\OO$ basepoints.  Then $${\Sf}({\bf x}_{\OO}) = q-1.$$
\end{lemma}

\begin{proof}
Take a fundamental domain for $G_K$  such that one of the $\OO$ basepoints is in the lower left-hand corner of the grid.\footnote{This choice has no effect on the computation of the absolute Maslov grading, as we will see; it is made only so that we can easily describe the procedure.}  If we forget about $\mathbb{X}$, we are left with an $n$--pointed Heegaard diagram for $L(p,q)$, where $n$ is the grid number of $G_K$.

Beginning with the $\beta$ circle to the right of this basepoint, let us now handleslide each of the $\beta$ circles over the one immediately to its right.  Do this until the $\beta$ circles on the diagram consist of a single curve of slope $-\frac{p}{q}$ and $n-1$ null-homotopic circles enclosing all but the left-most $O_i$.  See Figure~\ref{fig:CanGen}.

\begin{figure}
\begin{center}
\input{Figures/CanGen-new.pstex_t}
\end{center}
\caption{A $4$--pointed Heegaard diagram for $L(5,2)$, before and after performing handleslides.}
\label{fig:CanGen}
\end{figure}
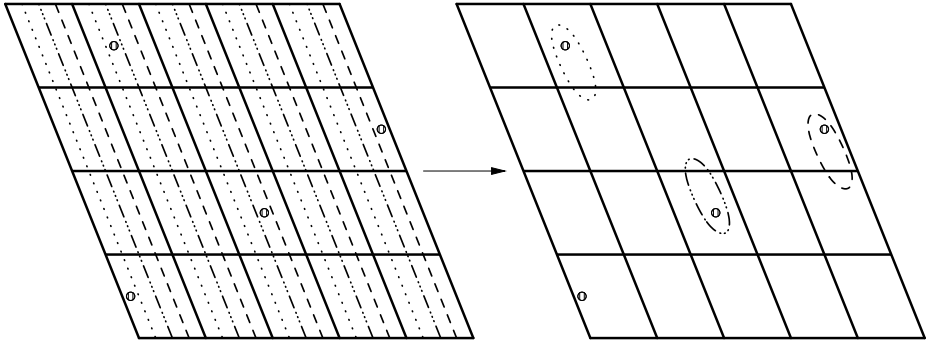

Now consider the Heegaard triple diagram pictured in Figure~\ref{fig:TripDiag} with $\mathbb{T}_{\alpha}$ specified by the original $\alpha$ curves, $\mathbb{T}_{\beta}$ specified by the original $\beta$ curves, and $\mathbb{T}_{\gamma}$ specified by the $\beta$ curves after we have performed the above sequence of handleslides.

\begin{figure}
\begin{center}
\input{Figures/TripDiag-new.pstex_t}
\end{center}
\caption{A Heegaard triple diagram showing a Maslov index $0$, $\Spin^c$ structure-preserving, triangle between the generator in $\mathbb{T}_\alpha \cap \mathbb{T}_\beta$ represented by the white circles and the generator in $\mathbb{T}_\alpha \cap \mathbb{T}_\gamma$ represented by the black circles.}
\label{fig:TripDiag}
\end{figure}
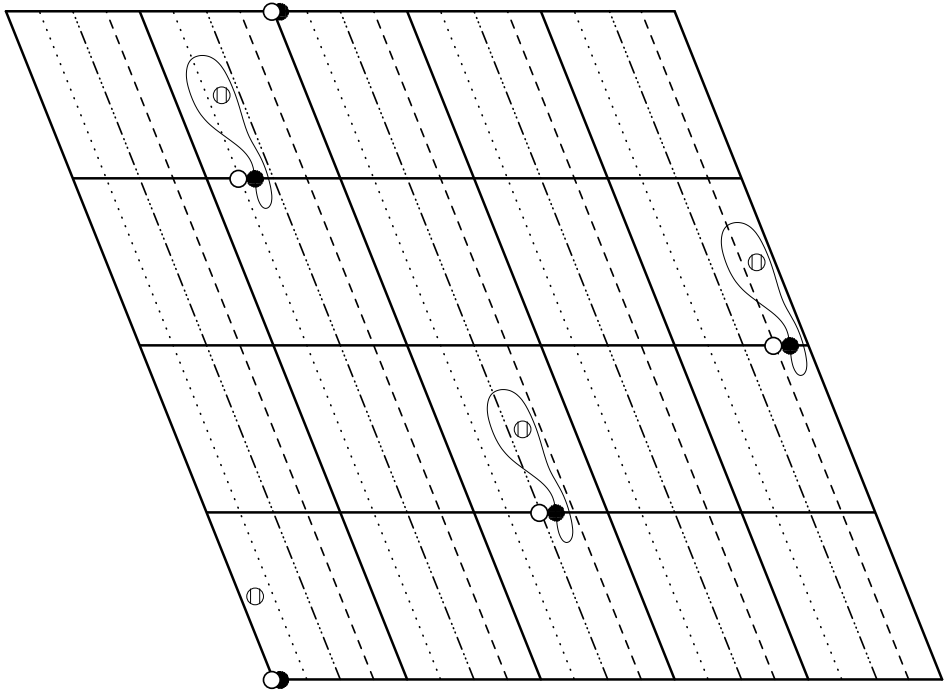

\begin{figure}
\begin{center}
\input{Figures/LowGen-new.pstex_t}
\end{center}
\caption{A generator in a $3$--times stabilized Heegaard diagram for $L(5,2)$.  It is the lowest Maslov index generator in $CF^-(L(p,q);\mathfrak{s}_{q-1})$ (using notation from \cite{MR1957829}), hence has absolute Maslov grading $d(L(p,q),\s_{q-1}) - (n-1)$.}
\label{fig:LowGen}
\end{figure}
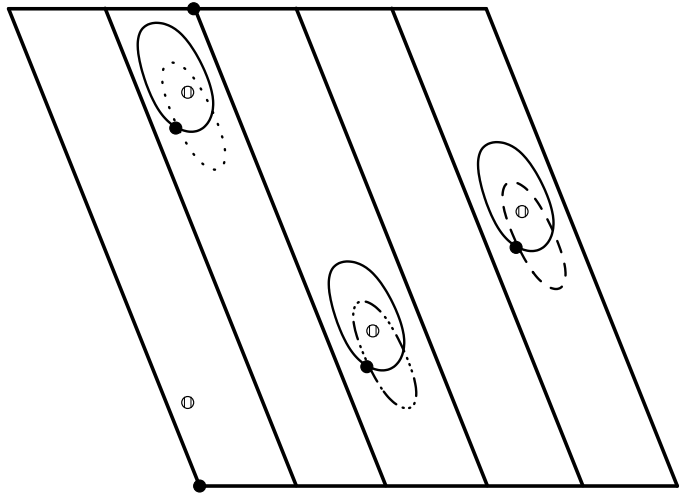

Let $\y$ be the generator pictured in Figure~\ref{fig:TripDiag} in $\mathbb{T}_{\alpha} \cap \mathbb{T}_{\gamma}$ by black circles.   We claim that $${\Sf}(\y)={\Sf}({\bf x}_{\OO}).$$  This follows from the existence of a Whitney triangle, $\psi$ connecting ${\bf x}_{\OO}$, $\y$, and a canonical generator $\Omega \in \mathbb{T}_{\beta} \cap \mathbb{T}_{\gamma}$.  Indeed, the four-dimensional cobordism, $W$, corresponding to the triple of curves $(\vec{\alpha},\vec{\beta},\vec{\gamma})$ is easily seen to be $L(p,q)\times [0,1]$ minus a regular neighborhood of the solid torus $V_\beta \times \frac{1}{2}$ (see Example ~8.1 of \cite{MR2113019}). It follows that the restriction of the $\Spin^c$ structure associated to $\psi$ to the $L(p,q)$ boundary components of $W$ must agree.  These restrictions, in turn, are $\mathfrak{s}_\OO(\y)$ and  $\mathfrak{s}_\OO({\bf x}_{\OO})$, respectively, proving the claim.

   On the other hand, ${\Sf}(\y)=q-1$,  according to the labeling convention specified in Section~4.1 of \cite{MR1957829}. Indeed,  $\y$ is the same generator pictured in Figure~\ref{fig:LowGen}, for which this last observation is obvious.
\end{proof}

\begin{lemma} \label{lemma:RelSpinc}  Let ${\bf x}_1=(\sigma_1,(b_0, \ldots, b_{n-1})), {\bf x}_2 = (\sigma_2, (c_0, \ldots, c_{n-1}))\in \mathcal{G}$.  Then
$${\Sf}({\bf x}_2) - {\Sf}({\bf x}_1) = \sum_{i=0}^{n-1} c_i - \sum_{i=0}^{n-1} b_i.$$  
\end{lemma}

\begin{proof}
Lemma~2.19 of \cite{MR2113019} indicates that 
\[{\Sf}({\bf x}_2) - {\Sf}({\bf x}_1)= \mathfrak{s}_{\mathbb{O}}({\bf x}_2) - \mathfrak{s}_{\mathbb{O}}({\bf x}_1) \in H_1(L(p,q))\cong \Z_p\]
 is represented by any cycle $\epsilon({\bf x}_2,{\bf x}_1)$ obtained by connecting ${\bf x}_2$ to ${\bf x}_1$ along the $\alpha$ curves and ${\bf x}_1$ to ${\bf x}_2$ along the $\beta$ curves. This number, in turn,  is the mod $p$ intersection of (a small transverse push off of) $\epsilon({\bf x}_2,{\bf x}_1)$ with any of the $\beta$ curves, say $\beta_0$.

Furthermore, transverse intersections with $\beta_0$ of a small horizontal push-off of $\epsilon({\bf x}_2,{\bf x}_1)$ occur only along the horizontal (i.e.\ $\alpha$) pieces of $\epsilon({\bf x}_2, {\bf x}_1)$.  Along a single $\alpha$ curve, the given labeling was chosen to count the number of intersections (mod $p$) of an arc of the $\alpha$ curve with $\beta_0$.

Therefore, the total number of intersections of $\epsilon({\bf x}_2,{\bf x}_1)$ with $\beta_0$ is 
$$\sum_{i=0}^{n-1} c_i - \sum_{i=0}^{n-1} b_i,$$ 
as desired.
\end{proof}

\begin{proof}[Proof of Proposition~\ref{prop:SpincGrad}]
It remains to show that $\Sp(\x) = {\Sf}(\x)$ for each generator $\x$.

With the definition of $\Sp$ preceding Proposition~\ref{prop:SpincGrad} observe that, given any 
\[{\bf x}_1=(\sigma_1,(b_0, \ldots, b_{n-1})), {\bf x}_2 = (\sigma_2, (c_0, \ldots, c_{n-1}))\in \mathcal{G},\]
  Lemma~\ref{lemma:RelSpinc} implies that $\Sp$ agrees with ${\Sf}$ up to an overall shift:
\[\Sp(\x_2) -  \Sp(\x_1) = \sum_{i=0}^{n-1} c_i - \sum_{i=0}^{n-1} b_i ={\Sf}(\x_2) - {\Sf}(\x_1).\]
Lemma~\ref{lemma:SpincCanGen} however shows that $\Sp$ agrees with ${\Sf}$ on the generator $\x_\OO$:
\[\Sp(\x_\OO) = q-1 = {\Sf}(\x_\OO).\]
Hence the two gradings agree for every generator.
\end{proof}

\subsubsection{Maslov gradings and the proof of Proposition~\ref{prop:MaslovGrad}.}
\begin{proof}[Proof of Proposition~\ref{prop:MaslovGrad}]
We now wish to show that the combinatorial quantity $\M$ agrees with the grading on knot Floer homology, $\gr$, defined in Subsection~\ref{subsec:GradDef}.
Our strategy will be to construct a Heegaard diagram compatible with the lift, $\widetilde{K}$, of $K$ in the universal cover of $L(p,q)$.  One easily constructs such a Heegaard diagram which, in fact, is a grid diagram for $\widetilde{K}$ in $S^3$ in the traditional sense \cite{GT0607691}.  A simple formula for the Maslov grading \cite{GT0610559} in the cover, coupled with the relative Maslov index formula of \cite{GT0608001} and a calculation of the absolute Maslov grading for a single generator completes the proof.

\begin{lemma}\label{lemma:CoverPicture} \cite{GT0701460} Let $T$ be a twisted toroidal grid diagram for $K$ in $L(p,q)$.  Form the universal cover 
$\mathbb{R}^2$ of $T$, identifying $T$ with the fundamental domain
$$[0,1] \times [0,1] \subset \mathbb{R}^2$$
of the covering space action.  Let $Z$ be the lattice generated by the vectors $(1,0)$ and $(0,p)$.  Then 
$$\widetilde{T} = \mathbb{R}^2/Z$$
 is a Heegaard diagram compatible with $\widetilde{K} \subset S^3$, where 
$\widetilde{K}$ is the preimage of $K$ under the covering space 
projection $\pi: S^3 \rightarrow L(p,q)$.\end{lemma}

Figure~\ref{fig:GenLift} illustrates a grid diagram for $\widetilde{K}$ in $S^3$ obtained from a grid diagram of $K$ in $L(p,q)$. 
Note that when $K$ is a knot, $\widetilde{K}$ will be a link of $\ell = \frac{p}{k}$ components, where $k$ is the order of $[K]$ as an element in $H_1(L(p,q);\Z)$.

The chain complex $CF^-(S^3,\widetilde{K})$ is generated by the points in $\mathbb{T}_{\widetilde{\alpha}} \cap \mathbb{T}_{\widetilde{\beta}}$, where $\widetilde{\alpha}$ and $\widetilde{\beta}$ are the lifts in the Heegaard diagram for $(S^3,\widetilde{K})$ of the $\alpha$ and $\beta$ curves for $(L(p,q),K)$.  To calculate the relative Maslov gradings between generators ${\bf x}, {\bf y} \in \mathbb{T}_{\alpha} \cap \mathbb{T}_{\beta}$, we use the natural map $$\mathbb{T}_{\alpha} \cap \mathbb{T}_{\beta} \rightarrow \mathbb{T}_{\widetilde{\alpha}} \cap \mathbb{T}_{\widetilde{\beta}}$$ which sends a generator ${\bf x} \in \mathbb{T}_{\alpha} \cap \mathbb{T}_{\beta}$ to $\widetilde{\bf x} = ({\bf x},\ldots,{\bf x}) \in \mathbb{T}_{\widetilde{\alpha}} \cap \mathbb{T}_{\widetilde{\beta}}$, the collection of its preimages, also depicted in Figure~\ref{fig:GenLift}.

\begin{figure}
\begin{center}
\input{Figures/GenLift-new.pstex_t}
\end{center}
\caption{A grid number $2$ diagram of a knot $K$ with a chain complex generator ${\bf x}$ in $L(3,1)$ and their lifts $\widetilde{K}$ and $\widetilde{\bf x}$ to a grid number $6$ diagram in the universal cover $S^3$.}
\label{fig:GenLift}
\end{figure}
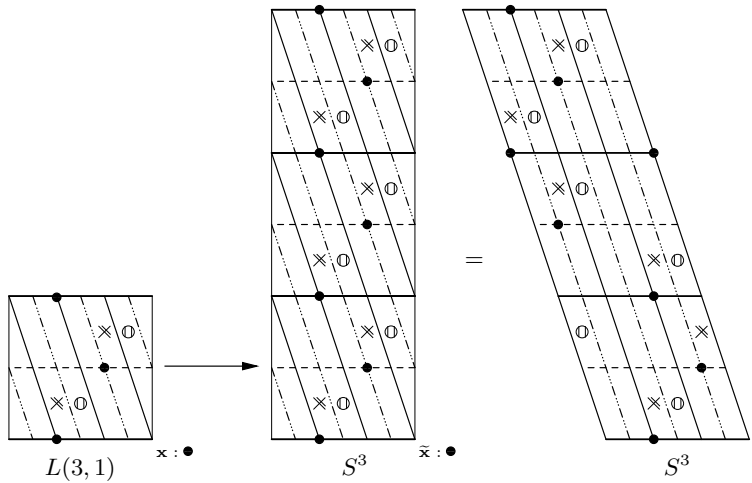

By \cite{GT0608001}, the Maslov grading differences between ${\bf x}, {\bf y} \in \mathcal{G}$ and $\widetilde{\bf x}, \widetilde{\bf y} \in \widetilde{\mathcal{G}}$ satisfy the relationship:  

\begin{equation} \label{eqn:LipLee}
\Mf(\x) - \Mf(\y) = \frac{1}{p}\left(\Mf(\widetilde{\x}) - \Mf(\widetilde{\y})\right).
\end{equation}

Furthermore, in \cite{GT0610559}, the authors provide a simple formula for the Maslov index of a generator in a toroidal grid diagram $G_K$ for a pair $(S^3,K)$ of a link $K$ in $S^3$.  In particular, given a point $a$ in a fundamental domain representing $G_K$, let $\pi_x(a)$ denote its $x$ (horizontal) coordinate and $\pi_y(a)$ denote its $y$ (vertical) coordinate.  They then define a function $\mathcal{I}(\,\cdot\, ,\, \cdot \, )$ whose input is two collections, $A = (a_1, \ldots, a_r)$ and $B=(b_1,\ldots,b_s)$ of finitely many coordinate pairs in $G_K \subset \mathbb{R}^2$, the chosen fundamental domain representing the Heegaard torus.  This function assigns to the pair $(A,B)$ the number of pairs $a \in A, b \in B$ satisfying $\pi_x(a) < \pi_x(b)$ and $\pi_y(a) < \pi_y(b)$.  They go on to show that for a generator ${\bf x} \in CF^-(S^3,K)$, 
$${\Mf}({\bf x}) = \mathcal{I}({\bf x},{\bf x}) - \mathcal{I}({\bf x},\mathbb{O}) - \mathcal{I}(\mathbb{O},{\bf x}) + \mathcal{I}(\mathbb{O},\mathbb{O}) + 1.$$  
In particular, ${\Mf}$ is independent of the chosen fundamental domain representing $G_K$.\footnote{We emphasize that the given formula requires a linear identification of $G_K$ with a fundamental domain on $\R^2$ with the property that $\alpha$ circles correspond to slope $0$ lines and $\beta$ circles to slope $\infty$ lines.  This is equivalent to what we have done: i.e., chosen coordinates on the fundamental domain $$\left\{(x,y) \in \R^2\,\,|\,\, 0 \leq y <p, -\frac{q}{p}y \leq x < -\frac{q}{p}y + 1\right\}$$ with respect to the basis $\vec{v}_1 = \left(\frac{1}{np},0\right), \vec{v}_2 = \left(-\frac{q}{np},\frac{1}{n}\right).$}
As in \cite{GT0610559}, we will find it convenient to use a bilinear extension of a symmetrized version of the $\mathcal{I}$ function:

$$\mathcal{J}(A,B) := \frac{1}{2}(\mathcal{I}(A,B) + \mathcal{I}(B,A)).$$

This allows the Maslov grading of a generator in an $S^3$ grid diagram to be expressed more succinctly as:

$$\Mf(\x) = \mathcal{J}(\x - \OO,\x-\OO) + 1.$$

\begin{lemma}  \label{lemma:AbsMasGr} Let $G_K$ be a grid number $n$ grid diagram for $K \subset L(p,q)$, and let ${\bf x}_{\mathbb{O}}$ denote the generator corresponding to the lower left corner of the $\mathbb{O}$ basepoints.  Then $${\Mf}({\bf x}_{\mathbb{O}}) = d(p,q,q-1) - (n-1).$$ 
\end{lemma}

Here $d(p,q,q-1)$ denotes the correction term $d(-L(p,q),q-1)$ as defined inductively in \cite{MR1957829}.  See Remark~\ref{remark:orientation}.

\begin{proof}
As in \cite{GT0607691}, we compute the absolute Maslov grading of the given generator by connecting it, via a Maslov index $0$ triangle in a Heegaard triple diagram, to a generator whose absolute Maslov grading we know.  See Section~8 of \cite{MR2113019} for the definition of a Heegaard triple diagram and Section~7 of \cite{MR2222356} for details on computing absolute Maslov gradings.

To do this, we use the same procedure as in the proof of Lemma~\ref{lemma:SpincCanGen}, noting that all triangles involved have Maslov index $0$.  See Figures~\ref{fig:CanGen},~\ref{fig:TripDiag},~\ref{fig:LowGen}.   It follows that  the generator pictured in Figure~\ref{fig:TripDiag} in $\mathbb{T}_{\alpha} \cap \mathbb{T}_{\gamma}$ by black circles has the same absolute Maslov grading as the generator pictured in Figure~\ref{fig:LowGen}.   This latter generator, in turn, has absolute Maslov grading $d(p,q,q-1)-(n-1)$, since it is the lowest generator in the $\Spin^c$ structure with $\Spin^c$ grading $q-1$ for an $n-1$ times stabilized Heegaard diagram.  See also the discussion in Section 2 of \cite{GT0607691}.  
\end{proof}

\begin{lemma} \label{lemma:AbsMasGrUp} Let $G_K$ be a grid number $n$ grid diagram for $K \subset L(p,q)$ and $G_{\widetilde{K}}$ the associated grid diagram for $\widetilde{K} \subset S^3$.  Then 
$${\Mf}(\widetilde{\bf x}_{{\OO}}) = -(pn-1).$$
\end{lemma}

\begin{proof}
$\widetilde{\bf x}_{\OO}$ is ${\bf x}_{\widetilde{\OO}}$, the generator in the lower-left hand corner of $\widetilde{\OO}$.  Thus one easily checks, by a calculation analogous to that detailed in the proof of Lemma 6.3 of \cite{GT0611841}, that ${\Mf}(\widetilde{\bf x}_{{\OO}}) = -(N-1)$, where $N$ is the grid number (see also Lemma~3.2 in \cite{GT0607691}).  In this case, $N=pn$.
\end{proof}

Thus, we arrive at Equation~\ref{equation:AbsGradShift}:
 
\begin{corollary} \label{cor:CompMas}  Let ${\bf x} \in \mathcal{G}$ be a generator associated to $G_K$ and $\widetilde{\bf x} \in \widetilde{\mathcal{G}}$ its lift in $G_{\widetilde{K}}$.  Then 
\begin{equation}\label{equation:AbsGradShift}
{\Mf}({\bf x}) = \frac{1}{p}{\Mf}(\widetilde{\bf x}) + \left(d(p,q,q-1) + \frac{p-1}{p}\right).
\end{equation}
\end{corollary}

\begin{proof}
One easily checks using Lemmas~\ref{lemma:AbsMasGr} and~\ref{lemma:AbsMasGrUp} that this formula holds for ${\bf x}_{\mathbb{O}}$.  The conclusion then follows from Equation~\ref{eqn:LipLee}.
\end{proof}

It is straightforward to verify that the function $\widetilde{W}$, defined in the discussion leading up to the statement of Proposition~\ref{prop:MaslovGrad}, sends a point in $G_K$ to the $p$--tuple of its preimages in $G_{\widetilde{K}}$. Then from Corollary~\ref{cor:CompMas} it follows that $\M(\x)=\gr(\x)$ for all $\x\in \mathcal{G}$, proving  Proposition~\ref{prop:MaslovGrad}.  Note that since $\gr = \M$ is a homological grading on $CF^-(L(p,q),K)$, it follows that $\M(\partial \x)=\M(\x)-1$.
\end{proof}
 
 \subsubsection{Alexander gradings and the proof of Proposition~\ref{prop:AlexGrad}.}
We conclude by identifying $\A$, as defined in Section~\ref{sec:combAlexGrad}, with the rational Alexander grading $\Af$.
\begin{proof}[Proof of Proposition~\ref{prop:AlexGrad}]
By Proposition~\ref{prop:MaslovGrad}, we have $\M=\gr$.  Therefore, it will be sufficient to show that the definition of the Alexander grading $\Af$ given in Subsection~\ref{subsec:GradDef} satisfies the stated relationship. That is, we must show
\begin{lemma}\label{lemma:OrReverse}
Let ${\bf x} \in \mathbb{T}_{\alpha} \cap \mathbb{T}_{\beta}$.  Then
   \begin{equation} \label{eqn:MandArelationship}
{\Af}_{\OO,\XX}({\bf x}) = \frac{1}{2}\left({\bf M}_{\OO}({\bf x}) - {\bf M}_{\XX}({\bf x}) - (n-1)\right)
   \end{equation}
\end{lemma}
\noindent  A proof of the lemma implies, in particular, that the combinatorial definition of $\A$ defines a filtration on $(CF^-(L(p,q)),\partial^-)$, since this is known to be true for the rational Alexander grading $\Af$.
\end{proof}

\begin{proof}[Proof of Lemma~\ref{lemma:OrReverse}]
The articles \cite{GT0607691} and \cite{GT0610559} give combinatorial descriptions of the Maslov and Alexander gradings for knots and links in $S^3$ and prove these combinatorial definitions match the original definitions.  Thus our strategy is to pass to the universal cover of $L(p,q)$ and build a grid diagram for the lift $\widetilde{K}\subset S^3$ of $K$ corresponding to $G_K$.  We then use the behavior of $\Mf$ and $\Af$ under covers to prove that the stated relationship holds for generators in the $2n$--pointed Heegaard diagram for $L(p,q)$.

If $[K] \in H_1(L(p,q);\Z)$ has order $k$, then $\widetilde{K}$ will be a link of $\ell = \frac{p}{k}$ components.   There is a natural map 
\[\mathbb{T}_{\alpha} \cap \mathbb{T}_{\beta} \rightarrow \mathbb{T}_{\widetilde{\alpha}} \cap \mathbb{T}_{\widetilde{\beta}}\]
 which lifts every generator to the $k$--tuple of its preimages  (cf.\ the discussion following Lemma~\ref{lemma:CoverPicture}).  Let $\mathcal{G}$ denote the set of generators $\mathbb{T}_\alpha \cap \mathbb{T}_\beta$ and similarly let $\widetilde{\mathcal{G}}=\mathbb{T}_{\widetilde{\alpha}} \cap \mathbb{T}_{\widetilde{\beta}}$.  
For ${\bf x} \in \mathcal{G}$, denote its lift by $\widetilde{\bf x} \in \widetilde{\mathcal{G}}$. 
Similarly, let ${\Mf}_{\widetilde{\OO}}$, ${\Mf}_{\widetilde{\XX}}$,  and ${\Af}_{\widetilde{\OO},\widetilde{\XX}}$ denote the Maslov and Alexander gradings relative to the lifted basepoints.

We now observe that if there exists a constant $\widetilde{C} \in \Z$  independent of $\widetilde{\x} \in \widetilde{\mathcal{G}}$, such that 
\[{\Af}_{\widetilde{\OO},\widetilde{\XX}}(\widetilde{\x}) = \frac{1}{2}\left({\Mf}_{\widetilde{\OO}}(\widetilde{\x}) - {\Mf}_{\widetilde{\XX}}(\widetilde{\x}) - (pn - 1)\right) + \widetilde{C}\]
 for all $\widetilde{\x} \in \widetilde{\mathcal{G}}$, then there exists a(nother) constant $C \in \Q$, independent of ${\x} \in \mathcal{G}$, such that 
 \[{\Af}_{\OO,\XX}({\x}) = \frac{1}{2}\left({\Mf}_{\OO}({\x}) - {\Mf}_{\XX}({\x}) - (n-1)\right) + C,\]
and hence by Proposition~\ref{prop:MaslovGrad},
 \[{\Af}_{\OO,\XX}({\x}) = \frac{1}{2}\left({\M}_{\OO}({\x}) - {\M}_{\XX}({\x}) - (n-1)\right) + C.\]

The observation follows immediately from the fact that the relative ${\Af}$ and ${\Mf}$ gradings  transform in the same way under the covering operation.  More precisely, \cite{GT0608001} tells us that for any two generators ${\x}, {\y} \in \mathcal{G}$, we have: 
\[{\Mf}_{\mathbb{O}}({\x}) - {\Mf}_{\mathbb{O}}({\y}) = \frac{1}{p}({\Mf}_{\widetilde{\mathbb{O}}}(\widetilde{\x}) - {\Mf}_{\widetilde{\mathbb{O}}}(\widetilde{\y}))\] 
and
\[{\Mf}_{\mathbb{X}}({\x}) - {\Mf}_{\mathbb{X}}({\y}) = \frac{1}{p}({\Mf}_{\widetilde{\mathbb{X}}}(\widetilde{\x}) - {\Mf}_{\widetilde{\mathbb{X}}}(\widetilde{\y})).\]
Similarly, Lemma~4.2 in \cite{GT0604360} says that: 
 \[{\Af}_{\mathbb{O},\mathbb{X}}({\x}) - {\Af}_{\mathbb{O},\mathbb{X}}({\y}) = \frac{1}{p}({\Af}_{\widetilde{\mathbb{O}},\widetilde{\mathbb{X}}}(\widetilde{\x}) - {\Af}_{\widetilde{\mathbb{O}},\widetilde{\mathbb{X}}}(\widetilde{\y})).\]

Thus
\[C = {\Af}_{\OO,\XX}({\x}) - \frac{1}{2}\left({\Mf}_{\OO}({\x}) - {\Mf}_{\XX}({\x}) - (n-1)\right)\]
 is independent of ${\x} \in \mathcal{G}$.

To prove  the existence of the $\widetilde{C}$ discussed above, we will appeal to Equation~2 from \cite{GT0610559}: 
\[A^\mathfrak{f}_i(\widetilde{\x}) = \mathcal{J}(\widetilde{\x} - \frac{1}{2}(\widetilde{\XX} + \widetilde{\OO})\,\,,\,\,\widetilde{\XX}_i - \widetilde{\OO}_i) - \left(\frac{n_i - 1}{2}\right).\]
See Subsection \ref{subsec:GradDef} for the definition of the components, $A^\mathfrak{f}_i$, of the Alexander multi-grading. Here, the notation ${\OO}_i$ (resp. ${\XX}_i$) refers to the subset of ${\OO}$ (resp. ${\XX}$) corresponding to the $i$th component of the link, and $\mathcal{J}$ is the bilinear extension of a symmetrized version of $\mathcal{I}$ which was defined in the discussion immediately preceding the statement of Lemma~\ref{lemma:AbsMasGr}.

Now suppose $\widetilde{K}$ has $\ell$ components.  We have: 
\begin{align*}
{\Af}_{\widetilde{\OO},\widetilde{\XX}}(\widetilde{\x}) &= \sum_{i=1}^\ell A^\mathfrak{f}_i(\widetilde{\x})\\
 &= \mathcal{J}(\widetilde{\x} - \frac{1}{2}(\widetilde{\XX} + \widetilde{\OO})\,\,,\,\,\sum_{i=1}^\ell(\widetilde{\XX}_i - \widetilde{\OO}_i)) - \left(\frac{pn-\ell}{2}\right)\\
 &= \mathcal{J}(\widetilde{\x} - \frac{1}{2}(\widetilde{\XX} + \widetilde{\OO})\,\,,\,\, \widetilde{\XX} - \widetilde{\OO}) - \left(\frac{pn-\ell}{2}\right)\\
 &= \mathcal{J}(\widetilde{\x}, \widetilde{\XX} - \widetilde{\OO}) - \frac{1}{2}[\mathcal{J}(\widetilde{\XX},\widetilde{\XX}) - \mathcal{J}(\widetilde{\OO},\widetilde{\OO})] - \left(\frac{pn - \ell}{2}\right)\\
\end{align*}
where equality holds from line 2 to line 3 because $\mathcal{J}(A,B+C) = \mathcal{J}(A, B \cup C)$ whenever $B \cap C = \emptyset.$

We now want to verify that this differs by a constant (independent of $\widetilde{\x}$) from:
\begin{align*}
\frac{1}{2}(\M_{\widetilde{\OO}}(\widetilde{\x}) - \M_{\widetilde{\XX}}(\widetilde{\x}) - (pn-1)) &= \frac{1}{2}(\mathcal{J}(\widetilde{\x} - \widetilde{\OO},\widetilde{\x} - \widetilde{\OO}) - \mathcal{J}(\widetilde{\x} - \widetilde{\XX},\widetilde{\x} - \widetilde{\XX}) - (pn-1))\\
 &= \frac{1}{2}[- 2\mathcal{J}(\widetilde{\x},\widetilde{\OO}) + \mathcal{J}(\widetilde{\OO},\widetilde{\OO}) + 2\mathcal{J}(\widetilde{\x},\widetilde{\XX}) -\mathcal{J}(\widetilde{\XX},\widetilde{\XX}) - (pn-1)]\\
 &= \mathcal{J}(\widetilde{\x},\widetilde{\XX} - \widetilde{\OO}) - \frac{1}{2}[\mathcal{J}(\widetilde{\XX},\widetilde{\XX}) - \mathcal{J}(\widetilde{\OO},\widetilde{\OO})] - \left(\frac{pn-1}{2}\right)\\
\end{align*}

It is now clear that $$\Af_{\widetilde{\OO},\widetilde{\XX}}(\widetilde{\x}) = \frac{1}{2}(\M_{\widetilde{\OO}}(\widetilde{\x}) - \M_{\widetilde{\XX}}(\widetilde{\x}) - (pn-1)) + \widetilde{C}$$ where $\widetilde{C} = \frac{\ell -1}{2},$ which does not depend on ${\widetilde{\x}}$.

It follows that Equation~\ref{eqn:MandArelationship} holds up to an overall shift by a constant, $C$.  But, by Lemma~\ref{lemma:Symmetry} we know that $${\Af}_{\OO,\XX}({\x}) = -{\Af}_{\XX,\OO}({\x}) - (n-1),$$ which forces $C = 0$.
\end{proof}

\end{document}

%% file: Figures/TwistedGrid-new2.pstex_t
\begin{picture}(0,0)%
\epsfig{file=TwistedGrid-new2.pstex}%
\end{picture}%
\setlength{\unitlength}{1973sp}%
\begingroup\makeatletter\ifx\SetFigFont\undefined%
\gdef\SetFigFont#1#2#3#4#5{%
  \reset@font\fontsize{#1}{#2pt}%
  \fontfamily{#3}\fontseries{#4}\fontshape{#5}%
  \selectfont}%
\fi\endgroup%
\begin{picture}(10824,8520)(-12611,-7669)
\put(-8249,464){\makebox(0,0)[b]{\smash{{\SetFigFont{12}{14.4}{\rmdefault}{\mddefault}{\updefault}{\color[rgb]{0,0,0}$\beta_2$}%
}}}}
\put(-11324,-2536){\makebox(0,0)[rb]{\smash{{\SetFigFont{12}{14.4}{\rmdefault}{\mddefault}{\updefault}{\color[rgb]{0,0,0}$\alpha_3$}%
}}}}
\put(-10724,-4036){\makebox(0,0)[rb]{\smash{{\SetFigFont{12}{14.4}{\rmdefault}{\mddefault}{\updefault}{\color[rgb]{0,0,0}$\alpha_2$}%
}}}}
\put(-10124,-5536){\makebox(0,0)[rb]{\smash{{\SetFigFont{12}{14.4}{\rmdefault}{\mddefault}{\updefault}{\color[rgb]{0,0,0}$\alpha_1$}%
}}}}
\put(-9524,-7036){\makebox(0,0)[rb]{\smash{{\SetFigFont{12}{14.4}{\rmdefault}{\mddefault}{\updefault}{\color[rgb]{0,0,0}$\alpha_0$}%
}}}}
\put(-4649,-1261){\makebox(0,0)[lb]{\smash{{\SetFigFont{12}{14.4}{\rmdefault}{\mddefault}{\updefault}{\color[rgb]{0,0,0}$R_3$}%
}}}}
\put(-4049,-2761){\makebox(0,0)[lb]{\smash{{\SetFigFont{12}{14.4}{\rmdefault}{\mddefault}{\updefault}{\color[rgb]{0,0,0}$R_2$}%
}}}}
\put(-3449,-4261){\makebox(0,0)[lb]{\smash{{\SetFigFont{12}{14.4}{\rmdefault}{\mddefault}{\updefault}{\color[rgb]{0,0,0}$R_1$}%
}}}}
\put(-2849,-5761){\makebox(0,0)[lb]{\smash{{\SetFigFont{12}{14.4}{\rmdefault}{\mddefault}{\updefault}{\color[rgb]{0,0,0}$R_0$}%
}}}}
\put(-5924,-7561){\makebox(0,0)[b]{\smash{{\SetFigFont{12}{14.4}{\rmdefault}{\mddefault}{\updefault}{\color[rgb]{0,0,0}$C_3$}%
}}}}
\put(-2849,-886){\rotatebox{270.0}{\makebox(0,0)[lb]{\smash{{\SetFigFont{12}{14.4}{\rmdefault}{\mddefault}{\updefault}{\color[rgb]{0,0,0}$x=1$}%
}}}}}
\put(-3449,-286){\makebox(0,0)[rb]{\smash{{\SetFigFont{12}{14.4}{\rmdefault}{\mddefault}{\updefault}{\color[rgb]{0,0,0}$y=1$}%
}}}}
\put(-11924,-1036){\makebox(0,0)[rb]{\smash{{\SetFigFont{12}{14.4}{\rmdefault}{\mddefault}{\updefault}{\color[rgb]{0,0,0}$\alpha_0$}%
}}}}
\end{picture}%

%% file: Figures/Generator-new.pstex_t
\begin{picture}(0,0)%
\epsfig{file=Generator-new.pstex}%
\end{picture}%
\setlength{\unitlength}{1381sp}%
\begingroup\makeatletter\ifx\SetFigFont\undefined%
\gdef\SetFigFont#1#2#3#4#5{%
  \reset@font\fontsize{#1}{#2pt}%
  \fontfamily{#3}\fontseries{#4}\fontshape{#5}%
  \selectfont}%
\fi\endgroup%
\begin{picture}(10329,8346)(739,-7519)
\put(751,-1111){\makebox(0,0)[rb]{\smash{{\SetFigFont{8}{9.6}{\rmdefault}{\mddefault}{\updefault}{\color[rgb]{0,0,0}$\alpha_0$}%
}}}}
\put(9151,-1561){\makebox(0,0)[b]{\smash{{\SetFigFont{8}{9.6}{\rmdefault}{\mddefault}{\updefault}{\color[rgb]{0,0,0}$3$rd point of}%
}}}}
\put(9976,-3961){\makebox(0,0)[b]{\smash{{\SetFigFont{8}{9.6}{\rmdefault}{\mddefault}{\updefault}{\color[rgb]{0,0,0}$\alpha_1 \cap \beta_0$}%
}}}}
\put(9976,-3511){\makebox(0,0)[b]{\smash{{\SetFigFont{8}{9.6}{\rmdefault}{\mddefault}{\updefault}{\color[rgb]{0,0,0}$2$nd point of}%
}}}}
\put(5701,-7111){\makebox(0,0)[b]{\smash{{\SetFigFont{8}{9.6}{\rmdefault}{\mddefault}{\updefault}{\color[rgb]{0,0,0}$\beta_0$}%
}}}}
\put(9601,-7411){\makebox(0,0)[b]{\smash{{\SetFigFont{8}{9.6}{\rmdefault}{\mddefault}{\updefault}{\color[rgb]{0,0,0}$4$th point of $\alpha_0 \cap \beta_2$}%
}}}}
\put(2851,539){\makebox(0,0)[b]{\smash{{\SetFigFont{8}{9.6}{\rmdefault}{\mddefault}{\updefault}{\color[rgb]{0,0,0}$4$th point of $\alpha_0 \cap \beta_2$}%
}}}}
\put(9151,-2011){\makebox(0,0)[b]{\smash{{\SetFigFont{8}{9.6}{\rmdefault}{\mddefault}{\updefault}{\color[rgb]{0,0,0}$\alpha_2 \cap \beta_1$}%
}}}}
\put(2401,-5161){\makebox(0,0)[rb]{\smash{{\SetFigFont{8}{9.6}{\rmdefault}{\mddefault}{\updefault}{\color[rgb]{0,0,0}$\alpha_1$}%
}}}}
\put(1576,-3211){\makebox(0,0)[rb]{\smash{{\SetFigFont{8}{9.6}{\rmdefault}{\mddefault}{\updefault}{\color[rgb]{0,0,0}$\alpha_2$}%
}}}}
\put(3151,-7111){\makebox(0,0)[rb]{\smash{{\SetFigFont{8}{9.6}{\rmdefault}{\mddefault}{\updefault}{\color[rgb]{0,0,0}$\alpha_0$}%
}}}}
\put(7051, 89){\makebox(0,0)[b]{\smash{{\SetFigFont{8}{9.6}{\rmdefault}{\mddefault}{\updefault}{\color[rgb]{0,0,0}$\beta_2$}%
}}}}
\put(5476, 89){\makebox(0,0)[b]{\smash{{\SetFigFont{8}{9.6}{\rmdefault}{\mddefault}{\updefault}{\color[rgb]{0,0,0}$\beta_1$}%
}}}}
\end{picture}%

%% file: Figures/ParGram-new2.pstex_t
\begin{picture}(0,0)%
\epsfig{file=ParGram-new2.pstex}%
\end{picture}%
\setlength{\unitlength}{1381sp}%
\begingroup\makeatletter\ifx\SetFigFont\undefined%
\gdef\SetFigFont#1#2#3#4#5{%
  \reset@font\fontsize{#1}{#2pt}%
  \fontfamily{#3}\fontseries{#4}\fontshape{#5}%
  \selectfont}%
\fi\endgroup%
\begin{picture}(8842,6344)(-7243,-6533)
\put(1126,-3736){\makebox(0,0)[rb]{\smash{{\SetFigFont{8}{9.6}{\rmdefault}{\mddefault}{\updefault}{\color[rgb]{0,0,0}${\bf{y}}:$}%
}}}}
\put(1126,-3286){\makebox(0,0)[rb]{\smash{{\SetFigFont{8}{9.6}{\rmdefault}{\mddefault}{\updefault}{\color[rgb]{0,0,0}${\bf{x}}:$}%
}}}}
\end{picture}%

%% file: Figures/Meridian.pstex_t
\begin{picture}(0,0)%
\epsfig{file=Meridian.pstex}%
\end{picture}%
\setlength{\unitlength}{1973sp}%
\begingroup\makeatletter\ifx\SetFigFont\undefined%
\gdef\SetFigFont#1#2#3#4#5{%
  \reset@font\fontsize{#1}{#2pt}%
  \fontfamily{#3}\fontseries{#4}\fontshape{#5}%
  \selectfont}%
\fi\endgroup%
\begin{picture}(5424,3558)(1189,-6244)
\put(5251,-3961){\makebox(0,0)[rb]{\smash{{\SetFigFont{6}{7.2}{\rmdefault}{\mddefault}{\updefault}{\color[rgb]{0,0,0}$\beta_{\ell}$}%
}}}}
\put(2251,-4711){\makebox(0,0)[rb]{\smash{{\SetFigFont{6}{7.2}{\rmdefault}{\mddefault}{\updefault}{\color[rgb]{0,0,0}$\beta_k$}%
}}}}
\put(3001,-3586){\makebox(0,0)[rb]{\smash{{\SetFigFont{6}{7.2}{\rmdefault}{\mddefault}{\updefault}{\color[rgb]{0,0,0}$\alpha_{j}$}%
}}}}
\put(4801,-4786){\makebox(0,0)[rb]{\smash{{\SetFigFont{6}{7.2}{\rmdefault}{\mddefault}{\updefault}{\color[rgb]{0,0,0}$\alpha_i$}%
}}}}
\put(3301,-4336){\makebox(0,0)[rb]{\smash{{\SetFigFont{6}{7.2}{\rmdefault}{\mddefault}{\updefault}{\color[rgb]{0,0,0}$\mu$}%
}}}}
\put(4201,-6061){\makebox(0,0)[b]{\smash{{\SetFigFont{6}{7.2}{\rmdefault}{\mddefault}{\updefault}{\color[rgb]{0,0,0}$K$}%
}}}}
\end{picture}%

%% file: Figures/CanGen-new.pstex_t
\begin{picture}(0,0)%
\epsfig{file=CanGen-new.pstex}%
\end{picture}%
\setlength{\unitlength}{1381sp}%
\begingroup\makeatletter\ifx\SetFigFont\undefined%
\gdef\SetFigFont#1#2#3#4#5{%
  \reset@font\fontsize{#1}{#2pt}%
  \fontfamily{#3}\fontseries{#4}\fontshape{#5}%
  \selectfont}%
\fi\endgroup%
\begin{picture}(16588,6088)(-643,-12705)
\end{picture}%

%% file: Figures/TripDiag-new.pstex_t
\begin{picture}(0,0)%
\epsfig{file=TripDiag-new.pstex}%
\end{picture}%
\setlength{\unitlength}{1381sp}%
\begingroup\makeatletter\ifx\SetFigFont\undefined%
\gdef\SetFigFont#1#2#3#4#5{%
  \reset@font\fontsize{#1}{#2pt}%
  \fontfamily{#3}\fontseries{#4}\fontshape{#5}%
  \selectfont}%
\fi\endgroup%
\begin{picture}(16888,12344)(-9615,-25733)
\end{picture}%

%% file: Figures/LowGen-new.pstex_t
\begin{picture}(0,0)%
\epsfig{file=LowGen-new.pstex}%
\end{picture}%
\setlength{\unitlength}{1973sp}%
\begingroup\makeatletter\ifx\SetFigFont\undefined%
\gdef\SetFigFont#1#2#3#4#5{%
  \reset@font\fontsize{#1}{#2pt}%
  \fontfamily{#3}\fontseries{#4}\fontshape{#5}%
  \selectfont}%
\fi\endgroup%
\begin{picture}(8488,6180)(8957,-5851)
\end{picture}%

%% file: Figures/GenLift-new.pstex_t
\begin{picture}(0,0)%
\epsfig{file=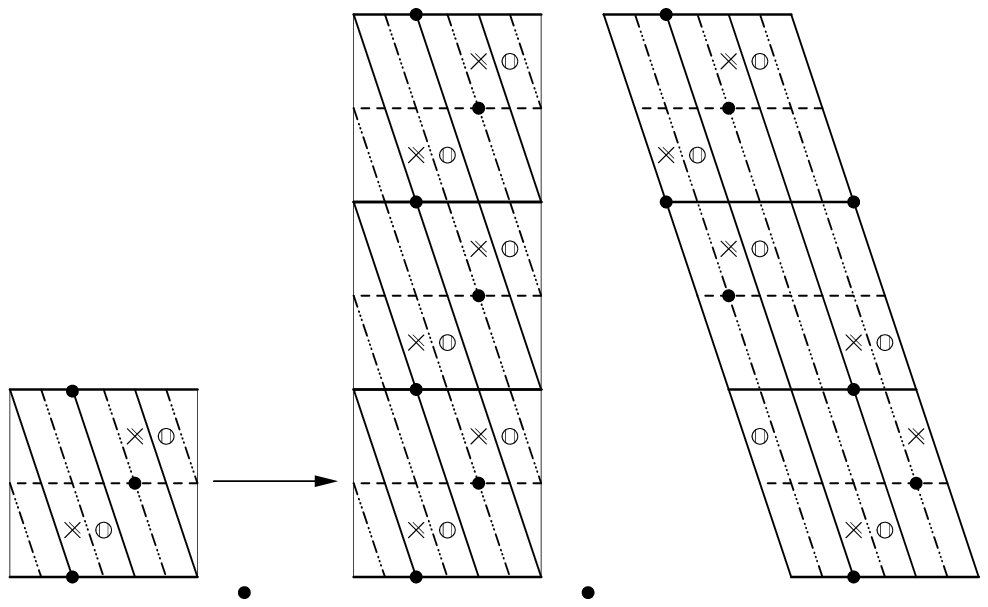}%
\end{picture}%
\setlength{\unitlength}{987sp}%
\begingroup\makeatletter\ifx\SetFigFont\undefined%
\gdef\SetFigFont#1#2#3#4#5{%
  \reset@font\fontsize{#1}{#2pt}%
  \fontfamily{#3}\fontseries{#4}\fontshape{#5}%
  \selectfont}%
\fi\endgroup%
\begin{picture}(18688,11976)(257,-8609)
\put(9001,-8461){\makebox(0,0)[b]{\smash{{\SetFigFont{9}{10.8}{\rmdefault}{\mddefault}{\updefault}{\color[rgb]{0,0,0}$S^3$}%
}}}}
\put(17101,-8461){\makebox(0,0)[b]{\smash{{\SetFigFont{9}{10.8}{\rmdefault}{\mddefault}{\updefault}{\color[rgb]{0,0,0}$S^3$}%
}}}}
\put(12001,-3211){\makebox(0,0)[b]{\smash{{\SetFigFont{9}{10.8}{\rmdefault}{\mddefault}{\updefault}{\color[rgb]{0,0,0}$=$}%
}}}}
\put(2101,-8461){\makebox(0,0)[b]{\smash{{\SetFigFont{9}{10.8}{\rmdefault}{\mddefault}{\updefault}{\color[rgb]{0,0,0}$L(3,1)$}%
}}}}
\put(11176,-8011){\makebox(0,0)[rb]{\smash{{\SetFigFont{6}{7.2}{\rmdefault}{\mddefault}{\updefault}{\color[rgb]{0,0,0}$\widetilde{{\bf x}}:$}%
}}}}
\put(4576,-8011){\makebox(0,0)[rb]{\smash{{\SetFigFont{6}{7.2}{\rmdefault}{\mddefault}{\updefault}{\color[rgb]{0,0,0}${\bf x}:$}%
}}}}
\end{picture}%